\documentclass[reqno,12pt]{amsart}
\usepackage{amsmath,amssymb,amsthm, verbatim}
\usepackage{url}
\usepackage[usenames, dvipsnames]{color}

\usepackage[letterpaper,hmargin=1in,vmargin=1in]{geometry}
\usepackage{graphicx}
\usepackage{enumerate,pinlabel}
\usepackage{mathrsfs,graphicx,url}
\usepackage[usenames,dvipsnames]{xcolor}
\usepackage[colorlinks=true,linkcolor=Black,citecolor=Black, urlcolor=Black]{hyperref}
\numberwithin{equation}{section}

\theoremstyle{plain}
\newtheorem{theorem}{Theorem}[section]

\newtheorem*{theorem*}{Theorem}
\newtheorem*{lemma*}{Lemma}
\newtheorem{lemma}[theorem]{Lemma}

\theoremstyle{definition}

\theoremstyle{remark}

\newtheorem{remark}[theorem]{Remark}

\newcommand{\ep}{\varepsilon}
\newcommand{\NN}{\mathbb{N}}
\newcommand{\R}{\mathbb{R}}
\newcommand{\US}{\mathbb{S}}
\newcommand{\C}{\mathbb{C}}

\newcommand\supp{\mathop{\rm supp}}

\newcommand\real{\mathop{\rm Re}}
\newcommand\imag{\mathop{\rm Im}}

\newcommand{\dist}{\mathop{\rm dist}}
\newcommand*{\defeq}{\mathrel{\vcenter{\baselineskip0.5ex \lineskiplimit0pt

                     \hbox{\scriptsize.}\hbox{\scriptsize.}}}%
                     =}
                     
\newcommand*{\qefed}{=\mathrel{\vcenter{\baselineskip0.5ex \lineskiplimit0pt

                     \hbox{\scriptsize.}\hbox{\scriptsize.}}}}

\begin{document}
\title[Resolvent bounds for repulsive potentials]{Resolvent bounds for repulsive potentials}

\author{Andr\'es Larra\'in-Hubach}
\address{Department of Mathematics, University of Dayton, Dayton, OH 45469-2316, USA}
\email{alarrainhubach1@udayton.edu}

\author{Yulong Li}
\address{Department of Mathematics, University of Dayton, Dayton, OH 45469-2316, USA}
\email{yli004@udayton.edu}
\author{Jacob Shapiro}
\address{Department of Mathematics, University of Dayton, Dayton, OH 45469-2316, USA}
\email{jshapiro1@udayton.edu}
\author{Joseph Tiller}
\address{Peerless Technologies, Beavercreek, OH 45324-2009, USA}
\email{joseph.tiller@epeerless.us}

\keywords{resolvent estimate, Schr\"odinger operator, repulsive potential, local energy decay}

\maketitle

\begin{abstract}
We prove limiting absorption resolvent bounds for the semiclassical Schr\"odinger operator with a repulsive potential in dimension $n\ge 3$, which may have a singularity at the origin. As an application, we obtain time decay for the weighted energy of the solution to the associated wave equation with a short range repulsive potential and compactly supported initial data. 
\end{abstract}

\section{Introduction and statement of results} \label{introduction}

The goal of this paper is to establish limiting absorption resolvent bounds for the semiclassical Schr\"odinger operator with a repulsive potential in dimensions three and higher. The dimension one case was studied in \cite[Section 2]{chda21}. As an application, we obtain time decay of a weighted energy for the solution to the associated wave equation with a short range repulsive potential and compactly supported initial data. 

We begin with some notation and conventions. Let $\Delta \defeq \sum_{j=1}^n \partial^2_j \le 0$ be the Laplacian on $\mathbb{R}^n$. We use $(r, \theta) = (|x|, x/|x|) \in (0, \infty) \times \US^{n-1}$ for polar coordinates on $\R^n \setminus \{0\}$. Throughout, we equip $\US^{n-1}$ with the standard Borel measure $d\theta$ such that the product measure $r^{n-1}dr \times d\theta$ gives Lebesgue measure on $(0,\infty) \times \US^{n-1}$. Put $B(0,r_0) \defeq \{x \in \R^n : |x| < r_0\}$. For a function $u$ defined on a subset of $\R^n$, we write $u(r, \theta) \defeq u(r \theta)$, and denote the radial derivative by $u' \defeq \partial_ru$. If $E \subseteq \R^n$ is a Borel set, $\mathbf{1}_E$ stands for its indicator function.

Our Schr\"odinger operator takes the form 
\begin{equation} \label{P}
P(h) \defeq -h^2 \Delta + V(x) : L^2(\mathbb{R}^n) \to L^2(\mathbb{R}^n),\qquad x \in \R^n,
\end{equation}
where $h > 0$ is the semiclassical parameter. The conditions we place on the Borel measurable potential $V : \R^n \to \R$ are as follows. We suppose $V=V^+-V^-$ with $V^+=\max(V,0)$, $V^-=-\min(V,0)$, and
\begin{gather}
V^-\in L^{\infty}(\mathbb{R}^n), \label{nonneg} \\  
r^{\rho(n)}V(x) \text{ is bounded for $r \le 1$,}  \label{for self-adjointness} \\
V(x) \text{ is bounded for $r \ge 1$,}  \label{for self-adjointness II} \\
\text{for each $\theta \in \US^{n-1}$, $(0, \infty) \ni r \mapsto V(r, \theta) \defeq V(r\theta)$ has bounded variation.} \label{AC loc} 
\end{gather}
Here, $\rho(n) > 0$ depends on the dimension $n$:
\begin{equation} \label{rho}
\rho(n) <  \begin{cases} \frac{3}{2} & n = 3, \\
2 & n \ge 4.
\end{cases}
\end{equation}

Recall that a function $f$ of locally bounded variation on an interval $I \subseteq \R$ has distributional derivative equal to a locally finite signed Borel measure, which we denote by $df$. In particular 
\begin{equation} \label{dist deriv}
\int \varphi df = - \int f \varphi'dx , \qquad \varphi \in C^\infty_0(I),
\end{equation}
where the $dx$ on the right side denotes Lebesgue measure on $I$; $df$ further satisfies,
\begin{equation} \label{ftc}
\int_{(a,b]} df = f^R(b) - f^R(a),
\end{equation}
for any interval $(a,b]$ contained in the interior of $I$, where $f^R(x) \defeq \lim_{\delta \to 0^+} f(x + \delta)$.

 The last condition we impose on $V$ is that there exists $C_V > 0$ so that for all $\theta \in \US^{n-1}$ and every bounded Borel set $E \subseteq (0, \infty)$,  
\begin{equation} \label{V prime cond}
\int_E dV( r ,\theta) \le -C_V\int_E (r +1 )^{-1} {V^+( r ,\theta)} dr.
\end{equation}
We emphasize that since the inequality \eqref{V prime cond} is one sided, each measure $dV(\cdot, \theta)$ is allowed to have negative point masses. Moreover, because only the positive part $V^+$ appears on the right side, the potential is allowed to approach a negative constant as $r \to \infty$. A prototype potential satisfying the above conditions is
\begin{equation*}
V(r, \theta) =g(\theta) \big( \mathbf{1}_{B(0,1)}  r^{-\rho(n)}  - 2^{-1} \mathbf{1}_{\R^n \setminus B(0,1)} (r^{-\delta} -2^{-1})\big)
\end{equation*}
for some $\delta > 0$ and $g\ge 0$ a bounded measurable function on $\US^{n-1}$. Note that for $V \in C^1(\R^n; [0, \infty))$, \eqref{V prime cond} implies each $V(\cdot, \theta)$ is repulsive in sense of classical mechanics, i.e., that $V(r, \theta) > 0$ implies $V'(r, \theta) < 0$. The local bound \eqref{for self-adjointness} allows for mild singularities at the origin, most notably the repulsive Coulomb behavior.  

For a real valued potential $V \in L^p(\R^n) + L^\infty(\R^n)$ with $p \ge 2$, $p > n/2$, $P(h)$ is self-adjoint if one takes the Sobolev space $H^2(\R^n)$ as the domain \cite[Theorem 8]{ne64}. The conditions \eqref{nonneg}, \eqref{for self-adjointness}, and \eqref{for self-adjointness II} imply that $V \in L^p(\R^n) + L^\infty(\R^n)$ for some such $p$. Our main results are the following weighted limiting absorption resolvent bounds for $P(h)$. 

 \begin{theorem} \label{newpreliminary} 
Suppose $n\ge 3$ and $V$ satisfies \eqref{nonneg} through \eqref{AC loc} and \eqref{V prime cond}. Define $P(h)$ by \eqref{P} and equip it with the domain $H^2(\R^n)$. For all $s > 1/2$ and $z=E\pm i\varepsilon$ with $E>0$ fixed, there is a constant $\mathfrak{C}(E,s,\varepsilon,\|V^-\|_{L^\infty})>0$ defined in \eqref{thm1.1aux} such that
\begin{equation}\label{thm1.1}
\|(r + 1)^{-s}(P(h)-z)^{-1}(r + 1)^{-s}\|_{L^2(\mathbb{R}^n)\to L^2(\mathbb{R}^n)}\leq \frac{\mathfrak{C}(E,s,\varepsilon,\|V^-\|_{L^\infty})}{h}.
\end{equation}
 \end{theorem} 
In case that $V^-=0,$ we prove stronger estimates.

\begin{theorem}  \label{nontrap est thm repulsive V}
Suppose $n \ge 3$ and  $V$ satisfies \eqref{nonneg} through \eqref{AC loc} and \eqref{V prime cond} with $V^-=0$. Define $P(h)$ by \eqref{P} and equip it with the domain $H^2(\R^n)$. For all $s, \, s_1, \, s_2 > 1/2$, with $s_1 + s_2 > 2$, there is $C> 0$ such that for all $z \in \C \setminus [0, \infty)$ and $h > 0$,
\begin{gather} 
 \| (r + 1)^{-s} (P(h) - z)^{-1} (r + 1)^{-s}  \|_{L^2(\R^n) \to L^2(\R^n)} \le  \frac{C}{h |z|^{1/2}},  \label{nontrap est repulsive V}\\
  \| (r + 1)^{-s_1} (P(h) - z)^{-1} (r + 1)^{-s_2}  \|_{L^2(\R^n) \to L^2(\R^n)} \le \frac{C}{h^2}. \label{low freq est repulsive V}
\end{gather} 
 \end{theorem} 
 
 \begin{remark}
From \eqref{resolv est in sector}, we see how the constant $C$ in \eqref{nontrap est repulsive V} depends on $s$ and $C_V$. The dependence of $C$ in \eqref{low freq est repulsive V} on $s_1$, $s_2$ and $C_V$ can be deduced from \eqref{rest}.
\end{remark}

Section \ref{nontrap resolv est section} is devoted to the proof of Theorem \ref{nontrap est thm repulsive V}, which builds on \cite[Theorem 1.2]{chda21}. In particular, Christiansen and Datchev obtained \eqref{nontrap est repulsive V} and \eqref{low freq est repulsive V} for bounded, repulsive potentials, nonnegative potentials on the half-line. The novelty of the present paper is that it extends these bounds to dimensions $n\ge 3$ for repulsive potentials that may be singular as $r \to \infty$. The corresponding problem in dimension two appears to be delicate, and to our knowledge remains open; we comment below on a possible approach.

\begin{remark} \label{sharpness remark}
In Appendix \ref{sharpness appendix}, we recall how for the case $V = 0$ and $n = 3$, the conditions on $s, \, s_1$ and $s_2$ in Theorem \ref{nontrap est thm repulsive V}, as well as the $h$- and $z$-dependencies of the right sides of \eqref{nontrap est repulsive V} and \eqref{low freq est repulsive V}, are nearly optimal in a suitable sense. 
\end{remark}

\begin{remark}
The weighted estimates underlying \eqref{thm1.1}, \eqref{nontrap est repulsive V}, and \eqref{low freq est repulsive V} hold under a condition weaker than \eqref{for self-adjointness}, namely that $| r^{n-1} V(r, \theta) | \to 0$ as $r \to 0$, see \eqref{Fr0} (the factor $r^{n-1}$ reflects the volume element of Lebesgue measure in polar coordinates). These estimates are derived for test functions in $C^\infty_0(\R^n)$ and are transferred to resolvent bounds via the density argument in Appendix~\ref{density appendix}. This step uses that the operator domain is $H^2(\R^n)$ and that $C^\infty_0(\R^n)$ is dense in $H^2(\R^n)$, which is why we impose the stronger hypothesis \eqref{for self-adjointness}. We note, however, that self-adjoint realizations of $-\Delta+V$ exist under weaker local assumptions on $V$. For example, if $V\in L^2_{\text{loc}}(\R^n)$ and $V\ge 0$, then $-\Delta+V$ is essentially self-adjoint on $C^\infty_0(\R^n)$ \cite[Theorem X.28]{resi75}; essential self-adjointness on $C^\infty_0(\R^n\setminus\{0\})$ holds under inverse square lower bounds \cite[Theorem X.30]{resi75}. It would be interesting to formulate the weighted estimates in such frameworks, which would allow more singular potentials. However, for the wave decay results discussed below, a restriction comes from the need to control $rV$ near the origin. We thus adopt the more streamlined approach here.

\end{remark}

We prove Theorems \ref{newpreliminary} and \ref{nontrap est thm repulsive V} using the \textit{spherical energy method}, a widely used technique for deriving weighted estimates for Schrödinger operators. The approach is based on separation of variables and the classical identity
\begin{equation} \label{conjugation 1}
    r^{\frac{n-1}{2}}(- \Delta) r^{-\frac{n-1}{2}} = -\partial^2_r + r^{-2} \Lambda,
\end{equation}
where 
\begin{equation} \label{Lambda}
    \Lambda \defeq -\Delta_{\US^{n-1}} + \frac{(n-1)(n-3)}{4},
\end{equation}
and $\Delta_{\US^{n-1}}$ denotes the negative Laplace-Beltrami operator on $\US^{n-1}$. The repulsivity condition is sufficiently advantageous to allow the use of a relatively simple weight--specifically, the same weight employed in \cite{chda21}--to obtain \eqref{est after integrating} and \eqref{uprimebound}. For more general potentials, it is usually necessary to instead conjugate the Laplacian by $e^{\varphi/h} r^{(n-1)/2}$ (see, e.g., \cite{cavo02, da14, gash22}) for a suitable phase $\varphi$. This results in a Carleman estimate with exponential losses as $h \to 0^+$. 

We use in a crucial way that $\Lambda \ge 0$ on $L^2(\US^{n-1})$, see \eqref{deriv wF lwr bd}. This is why our approach does not cover the case $n = 2$ where the effective potential $-1/(4r^2)$ has a strong negative singularity as $r \to 0$. We expect that repulsive potentials in dimension two can be treated by adapting the Mellin transform methods used in \cite{dgs23, ob24}.

As an application of \eqref{low freq est repulsive V}, we prove weighted energy decay for the solution to the wave equation
\begin{equation} \label{wave eqn srp}
\begin{cases}
(\partial^2_t - \Delta + V(x))u(t,x) = 0 & (t, x) \in \R \times \R^n, \, n \ge 3, \\
u(0,x) = u_0(x), \, \partial_tu(0,x) = u_1(x) & x \in \R^n,
\end{cases}
\end{equation}
where $u_0 \in H^1(\R^n)$ and $u_1 \in L^2(\R^n)$ have compact support. The potential obeys $V \ge 0$, \eqref{for self-adjointness}, \eqref{AC loc} with $\rho = \rho(n) = 1$, \eqref{V prime cond}, and the extra short range condition
\begin{equation} \label{stronger bd V}
\mathbf{1}_{\R^n \setminus B(0,1)} V  \le C(r + 1)^{-\delta(n)}, 
\end{equation}
for some $C > 0$ and $\delta(n) > 0$ such that
\begin{equation} \label{delta restrictions}
\delta(n) > \begin{cases}
\frac{1}{2} + \frac{n + 3}{4}  & n \neq 8, \\
\frac{1}{2} + 3 & n = 8. 
\end{cases}
\end{equation}
Since $P \defeq P(1) = -\Delta + V$ is self-adjoint (and nonnegative) under such conditions, we may use the spectral theorem for self-adjoint operators to represent the solution to \eqref{wave eqn srp} by
\begin{equation} \label{soln spect thm}
u(t, \cdot) = \cos(t \sqrt{P}) u_0 + \frac{\sin(t \sqrt{P})}{\sqrt{P}}u_1.
\end{equation}
For $s > 0$ fixed, define the weighted energy of the solution $u$ to \eqref{wave eqn srp} to be 
\begin{equation*}
E_s[u](t) = E_s(t) \defeq  \int_{\R^n} \langle x \rangle^{-2s}( |\partial_t u(t,x)|^2 + | \nabla u(t,x)|^2 + |u(t,x)|^2 )dx.
\end{equation*}
Set also
\begin{equation*}
E(0) \defeq  \|\nabla u_0 \|^2_{L^2} + \| u_1 \|^2_{L^2}.
\end{equation*}

\begin{theorem} \label{wave decay thm}
Suppose $V \ge 0$ satisfies \eqref{for self-adjointness}, \eqref{AC loc} with $\rho = 1$, \eqref{V prime cond}, as well as \eqref{stronger bd V}. Let $u_0 \in H^1(\R^n)$ and $u_1 \in L^2(\R^n)$ have compact support. For each $s$ such that
\begin{equation} \label{s restrictions}
s > \begin{cases}
\frac{n + 3}{4}  & n \neq 8, \\
3 & n = 8,
\end{cases}
\end{equation}
there exists $C_s > 0$ depending on $s$ but independent of $t$, $u_0$, and $u_1$ so that 
\begin{equation} \label{LED}
E_s(t) \le C_s\langle t \rangle^{-2} E(0),
\end{equation}
where $\langle t  \rangle \defeq (1 + |t|^2)^{1/2}$.
\end{theorem}

\begin{remark}
Since $u(-t, \cdot) = \cos(t \sqrt{P}) u_0 +  (\sin(t \sqrt{P})/\sqrt{P})(-u_1)$, to prove \eqref{LED} it suffices to suppose $t \ge 0$.
\end{remark}

\begin{remark}
We expect that, by adapting the arguments of \cite[Section 3]{vo04b}, the assumption of compact support on the initial data can be relaxed. Specifically, the decay should continue to hold for initial conditions lying in an appropriate weighted Sobolev spaces. However, to keep the presentation technically streamlined, we suppose $u_0$ and $u_1$ have compact support. 
\end{remark}

\begin{remark}
The condition \eqref{delta restrictions} has a quirk in dimension eight compared to other dimensions. This appears to be an artifact of our approach, which analyzes the low-frequency behavior of the resolvent kernel in terms of Hilbert–Schmidt norms (Appendix \ref{deriv free resolv appendix}).
\end{remark}

For smooth, nonnegative potentials of compact support, the local energy
\begin{equation*}
E_{r_0}(t) \defeq \int_{B(0,r_0)}  |\partial_t u(t,x)|^2 + | \nabla u(t,x)|^2 + |u(t,x)|^2 dx, \qquad r_0 > 0,
\end{equation*}
obeys
\begin{equation} \label{compact support case}
E_{r_0}(t) = \begin{cases} 
O(e^{-ct}) \text{ for some } c> 0 & n \ge 3 \text{ odd}, \\
O(t^{-2n}) & n \ge 4 \text{ even}.
\end{cases}
\end{equation}
Indeed, Vainberg \cite{va75} showed \eqref{compact support case} for compactly supported perturbations of the Laplacian satisfying the so-called Generalized Huygens Principle (as defined in \cite{vo04c}). A result of Melrose and Sjöstrand on the propagation of singularities \cite{mesj78, mesj82} implies that this principle holds for a broad class of nontrapping perturbations of the Laplacian, which includes smooth, nonnegative, compactly supported potentials. The study of energy decay for nontrapping perturbations has a long history, tracing back to the works of Lax, Morawetz, and Phillips \cite{lmp63, mo66, mo75}.

On the other hand, bounds similar to \eqref{LED} were obtained in previous works for various classes of short range potentials. In \cite{za04}, Zappacosta considered potentials $V \in C^1(\R^3; (0,\infty))$ with $\partial^\alpha_xV = O(\langle x \rangle^{- \delta - |\alpha|})$ for all $0 \le |\alpha| \le 1$ and some $\delta > 2$. For each $\chi \in C^\infty_0(\R^3)$, the bound \\$\| \chi \sqrt{V}(\sin(t \sqrt{P})/\sqrt{P}) \sqrt{V} \chi \|^2_{L^2 \to L^2} = O(t^{-2})$ was proved. In \cite{vo04}, Vodev showed $E_{r_0}(t) = O(t^{-2})$ in dimension $n \ge 3$, where $V \in C^1(\R^n ; [0,\infty))$ obeys  
\begin{gather} 
 V = O(\langle x \rangle^{-\delta_0}) \qquad \text{for some $\delta_0 > 2$, and} \label{vodev cond 1} \\
2V + r \partial_rV \le C\langle x \rangle^{-\delta} \qquad \text{for some $C > 0$ and some $\delta > 1$.} \label{vodev cond 2}
\end{gather}

Additionally, it is assumed that $V$ has no resonance at zero energy, a condition closely related to the validity of a bound such as \eqref{low freq est repulsive V} for $|z|\ll1$. Vodev also obtained weighted energy decay for a class of long-range, nontrapping perturbations of the Laplacian that includes perturbations by a nonnegative long-range potential, provided the initial conditions are spectrally localized away from $[0,a]$ for $a>0$ sufficiently large \cite{vo04b}. We note that, in our approach, the positivity of the potential appears crucial for obtaining \eqref{low freq est repulsive V}. Indeed, the constant $\mathfrak{C}(E,s,\varepsilon,\|V^-\|_{L^\infty})$ in \eqref{thm1.1aux} blows up as $E\to0^+$ unless $\|V^-\|_{L^\infty}=0$.

We note the connection between \eqref{vodev cond 2} and our repulsiveness condition. If \\ $V \in C^1(\R^n ; (0,\infty))$ satisfies \eqref{V prime cond}, then $\partial_r (\log V(\cdot , \theta )) \le -C_V(r + 1)^{-1}$ uniformly in $\theta$. Integration in $r$ yields $0 \le V(r, \theta) \le C(r + 1)^{-C_V}$, so $2V + r\partial_r V \le 2V \le C(r + 1)^{-C_V}$, which is \eqref{vodev cond 2} provided $C_V > 1$.

In the absence of a bound like \eqref{low freq est repulsive V}, or a condition excluding a resonance or eigenvalue at zero, decay of wave equation solutions generally cannot be expected. Thus, in the present work the positivity of the potential serves as a sufficient condition to rule out threshold obstructions. Other sufficient conditions have been obtained in previous works; for example, smallness of the potential in the Rollnik and global Kato norms \cite{rosc04}. For related discussions and further references, see \cite{jene01, chda25,cdy25}.

The proof of Theorem \ref{wave decay thm} follows the strategy of \cite[Section 3]{vo04}, with modifications to account for the possible singularity of $V$ at the origin. The key step is to establish
\begin{equation} \label{wave decay outline}
 t^2 E_s(t) \le C t^2 \int_t^\infty E_s(\tau) d\tau \le CE(0), \qquad t \ge 1.
\end{equation}
Using Duhamel's formula and the Fourier transform ($t$ dual to $\lambda$), the Fourier transform of $u$ can be expressed in terms of $(P - \lambda^2)^{-1}$, see \eqref{FT identity}. Because the initial data are compactly supported, finite speed of propagation allows insertion of a cutoff function $\eta$, and Plancherel's theorem then reduces control of $E_s(t)$, to bounds on $\langle x \rangle^{-s}(P - \lambda^2)^{-1}\eta$, which comes from \eqref{low freq est repulsive V}, see also Lemma \ref{lap lemma}. 

However, the factor $t^2$ in \eqref{wave decay outline} corresponds to differentiation with respect to $\lambda$, so it is also necessary to control  $\langle x \rangle^{-s} \tfrac{d}{d\lambda}(P - \lambda^2)^{-1}\eta $. This is achieved under the stronger assumptions \eqref{for self-adjointness} with $\rho = 1$ (compare with \eqref{rho}) and \eqref{stronger bd V}. In particular, the restriction $\rho = 1$ arises from the need for a uniform bound on $\lambda V (-\Delta - \lambda^2)^{-1} \langle x \rangle^{-s}$ for $\imag \lambda > 0$, see \eqref{deriv resolv id}. It would be interesting to determine whether more singular potentials could be accommodated by controlling the $\lambda$-derivative in a less perturbative manner.

We expect that for certain potentials with mild spatial decay depending on the dimension, $t^{-2}$ is the optimal decay for the local energy. A sharper description should depend on a low-frequency expansion of the resolvent around $\lambda = 0$ (see, e.g., \cite{jene01}), rather than a bound alone. 

Another energy studied is the quantity
\begin{equation*}
E^{(1)}_K[u](t) \defeq \int_{K}  |\partial_t u(t,x)|^2 + | \nabla u(t,x)|^2 + V(x)|u(t,x)|^2 dx,
\end{equation*}
where $K \subseteq \R^n$ is a region of interest. In \cite[Theorem 1.1]{vo04}, Vodev studied the case $K = B(0, \gamma_0 t) \subseteq \R^n$ for $n \ge 3$ and suitable $0 < \gamma_0 < 1$. Under the assumption of no resonance at zero, along with \eqref{vodev cond 1} for some $\delta_0 > 1$ and \eqref{vodev cond 2} for constants $C> 0$ and $\delta > 1$, he proved $E^{(1)}_K[u](t) = O(t^{-1})$. In \cite[Theorems 1.1 and 1.2]{ik23}, Ikehata considered exterior subdomains $\Omega$ of $\R^n$, $n \ge 2$, excluding the origin. For  compact subsets $K \subseteq \Omega$, he established the same decay rate assuming $V$ is nonnegative, $C^1$, and obeys $x\cdot \nabla V + 2V \le 0$. The $O(t^{-1})$-decay was first showed by Morawetz \cite{mo61} for $V = 0$ in the exterior of a three-dimensional star-shaped obstacle (later improved to exponential decay in \cite{lmp63}).

There is an extensive body of literature on wave decay for higher-order perturbations. For general second-order perturbations that may exhibit trapping, logarithmic decay--rather than polynomial decay--is more typical. For historical background and related developments, see \cite{bu98, vo99, bu02, bo11, cavo04, sh18, chik20}.

The rest of the paper is organized as follows. In Section \ref{nontrap resolv est section}, we prove Theorem \ref{nontrap est thm repulsive V}. In Section \ref{Helmholtz resolvent section}, we apply Theorem \ref{nontrap est thm repulsive V} to prove norm bounds for $\langle x \rangle^{-s}(P - \lambda^2)^{-1}\langle x \rangle^{-s}$ and its $\lambda$-derivative. In Section \ref{wave decay section} we prove Theorem \ref{wave decay thm}. Finally, we include several appendices of technical results that assist with the proofs of earlier sections.

\medskip
\noindent{\textsc{Acknowledgements:}} We thank Kiril Datchev and Georgi Vodev for helpful discussions. We also thank the anonymous referee for careful reading and helpful comments that improved the paper. J. S. and A. L.-H. gratefully acknowledge support from NSF DMS-2204322. J. S. was also supported by a University of Dayton Research Council Seed Grant.

 \section{Proof of Theorem \ref{nontrap est thm repulsive V}} \label{nontrap resolv est section}
In this section, we prove Theorem \ref{nontrap est thm repulsive V}. Throughout this section, we take $P(h)$ as in \eqref{P}, and assume the potential $V$ satisfies \eqref{nonneg} through \eqref{AC loc} and\eqref{V prime cond}.

By \eqref{conjugation 1},
\begin{equation} \label{conjugation 2}
\begin{split}
  P^{\pm}(h) &\defeq  r^{\frac{n-1}{2}}\left( P(h) - E \pm i\varepsilon \right) r^{-\frac{n-1}{2}}\\
  &= -h^2\partial^2_r + h^2r^{-2} \Lambda + V  - E \pm i\varepsilon,
 \end{split}
\end{equation}
where we let $E$ and $\ep$ vary in $[0, \infty)$. Let $u \in r^{(n-1)/2} C^\infty_0(\R^n)$. Define a spherical energy functional $F[u](r)$,
\begin{equation} \label{F}
    F(r) = F[u](r) \defeq \|hu'(r, \cdot)\|^2 - \langle (h^2 r^{-2} \Lambda + V^R(r, \cdot)  - E)u(r, \cdot), u(r, \cdot) \rangle,
\end{equation}
where $\| \cdot \|$ and $\langle \cdot, \cdot \rangle$ denote the norm and inner product on $L^2(\mathbb{S}_\theta^{n-1})$.   We take complex conjugation to occur in the first argument of $\langle \cdot, \cdot \rangle$. Here, $V^R(r, \theta)$ is the measurable function defined by $V^R(r, \theta) \defeq \lim_{k \to \infty} V((r + (1/k))\theta)$ for $k \in \NN$. The limit exists for each $r$ and $\theta$ since each $V(\cdot, \theta)$ has bounded variation. Each $V^R(\cdot, \theta)$ is decreasing thanks to \eqref{nonneg}, \eqref{ftc} and \eqref{V prime cond}.

 For a weight $w(r)$ which is absolutely continuous, nonnegative, and increasing, we compute the derivative of $wF$ in the sense of distributions on $(0, \infty)$. For this we need the following technical lemma whose proof we give in Appendix \ref{technical lemma appendix}.
 \begin{lemma} \label{technical lemma}
  The function $r \mapsto \int_{\US^{n-1}} V^{R}(r, \theta) |u(r, \theta)|^2 d\theta$ has locally bounded variation and its derivative in the sense of distributions on $(0, \infty)$ is given by
  \begin{equation} \label{technical dist formula}
  \begin{split}
  C^\infty_0(0, \infty) \ni \varphi &\mapsto \int_0^\infty \varphi(r) \int_{\US^{n-1}} V(r, \theta) 2 \real(\overline{u}u')  d\theta dr\\
  &+ \int_{\US^{n-1}}  \int_0^\infty  \varphi(r) |u(r, \theta)|^2 dV(r,\theta) d\theta.
  \end{split}
  \end{equation} 
 \end{lemma}

Note that, aside from the term in $F(r)$ involving $V^R(r, \theta)$, it is clear the remaining terms are functions of $r$ with locally bounded variation on $(0, \infty)$. Thus by Lemma \ref{technical lemma}, $F(r)$ itself is of locally bounded variation on $(0, \infty)$.
 
 Using \eqref{technical dist formula} we have, in the sense of distributions on $(0,\infty)$:
\begin{equation} \label{deriv wF}
\begin{split}
    (wF)' &=  wF' + w'F\\
    &= w(-2\real \langle (-h^2\partial^2_r + h^2r^{-2}\Lambda + V - E)u  , u' \rangle \\
     &+2 h^2 r^{-3} \langle \Lambda u, u \rangle  - \textstyle\int_{\US^{n-1}} |u(r, \theta)|^2 dV(r, \theta) d\theta )   \\
     &+ w' (\|hu'\|^2 - \langle h^2 r^{-2} \Lambda  u, u \rangle + \langle (E - V)u, u \rangle\\
    &= -2 w \real \langle P^{\pm}(h) u, u' \rangle \pm 2\varepsilon w \imag \langle u,u'\rangle + w'\|hu'\|^2\\ 
    &+(2wr^{-1} - w') \langle h^2r^{-2}\Lambda u,u\rangle  + Ew'\| u\|^2 \\
    &-\textstyle\int_{\US^{n-1}} |u(r, \theta)|^2 (w(r) dV(r, \theta) + w'(r)V(r, \theta)) d\theta  . 
    \end{split}
\end{equation}

First we show \eqref{nontrap est repulsive V}. Since increasing $s$ decreases the left side of \eqref{nontrap est repulsive V}, without loss of generality we may take $0 < \delta \defeq 2s - 1 < 1$. We will show the last line of \eqref{deriv wF} can be made to have a suitable lower bound, using 
\begin{equation} \label{w general repulsive}
w(r) \defeq 1 - \frac{C_V}{C_V + \delta}(1 + r)^{-\delta}.
\end{equation}
For such $w$, we clearly have
\begin{equation*}
w'(r) = \frac{\delta C_V}{C_V + \delta}(r + 1)^{-1 - \delta}.
\end{equation*}
Therefore, on the one hand
\begin{equation*}
2wr^{-1} - w' = 2r^{-1} (r +1)^{-\delta}\big((r +1)^\delta - \frac{C_V}{C_V + \delta} \big[ 1 +  \frac{\delta r}{2(r + 1)}  \big]   \big) \ge 0,
\end{equation*}
where we used 
\begin{equation*}
(r +1)^{\delta} -  \frac{C_V}{C_V + \delta} \ge  \frac{\delta C_V}{C_V + \delta} \int_1^{r +1} s^{\delta - 1} ds \ge  \frac{\delta C_V r}{(C_V + \delta)(r+1)},
\end{equation*}
since $\delta < 1$. On the other hand, using \eqref{V prime cond}, we have, in the sense of measures on bounded Borel subsets of $(0, \infty)$,
\begin{equation*}
wdV + w'V  = \frac{\delta C_V V}{(C_V + \delta)(r + 1)^{1 + \delta}} + wdV \le \frac{C_V V^+}{1 + r} ( (r +1 )^{-\delta} - 1 ) \le 0.
\end{equation*}

Thus, the last two estimates and \eqref{deriv wF} imply, for any $\varphi \in C^\infty_0((0,\infty); [0, \infty))$

\begin{equation} \label{deriv wF lwr bd}
\begin{split}
\int \varphi d(wF) &=  -\int (wF) \varphi'dr\\
 &\ge \int \big( -2 w \real \langle P^{\pm}(h) u, u' \rangle \pm 2\varepsilon w \imag \langle u,u'\rangle \\ 
&+ w' \| hu'\|^2 + E w'\| u\|^2\big) \varphi dr. 
\end{split}
\end{equation}
Note that in the first line of \eqref{deriv wF lwr bd} we applied \eqref{dist deriv}. Now, take a sequence of $\varphi_k \in C^\infty_0((0,\infty); [0, 1])$ that converges pointwise to the indicator function $\mathbf{1}_{(r_0, r_1]}$ with $0 < r_0 \ll 1$ and $r_1$ large enough so that $u(r, \theta) = 0$ for near $[r_1, \infty) \times \US^{n-1}$. Substituting $ \varphi = \varphi_k$ in \eqref{deriv wF lwr bd}, sending $k \to \infty$, and applying the dominated convergence theorem and \eqref{ftc} gives 
\begin{equation} \label{pre est after integrating}
\int_{r_0}^\infty E w' \| u\|^2 + w' \| hu'\|^2 dr  + w(r_0) F^R(r_0) \le  \int_{r_0}^\infty 2 w \real \langle P^{\pm}(h) u, u' \rangle \mp 2\varepsilon w \imag \langle u,u'\rangle dr.
\end{equation} 

Since $u = r^{(n-1)/2} v$ for some $v \in C^\infty_0(\R^n)$, we recognize that 
\begin{equation} \label{Fr0}
\begin{split}
F^R(r_0) &= \|hu'(r_0, \cdot) \|^2 +  r_0^{n-3} \langle h^2  \Delta_{\US^{n-1}}v(r_0, \cdot), v(r_0, \cdot) \rangle \\
&+(Er^{n-1}_0 - h^24^{-1}(n-1)(n-3)r_0^{n-3}) \|v(r_0, \cdot)\|^2 + r_0^{n-1} \langle V^R(r_0, \cdot)v(r_0, \cdot), v(r_0, \cdot) \rangle.
\end{split}
\end{equation}
We rewrite the term in \eqref{Fr0} involving $\Delta_{\US^{n-1}}$ using the well known formula for the Laplacian in spherical coordinates:
\begin{equation*}
r^{-2}\Delta_{\US^{n-1}} = \Delta - \partial^2_r - (n-1)r^{-1} \partial_r. 
\end{equation*}
Therefore, 
\begin{equation} \label{rewrite laplace beltrami}
\begin{split}
r_0^{n-3}& \langle h^2  \Delta_{\US^{n-1}}v(r_0, \cdot), v(r_0, \cdot) \rangle\\
&= h^2 r_0^{n-1}\langle (\Delta v)(r_0, \cdot), v(r_0, \cdot)  \rangle \\
&- h^2r_0^{n-1}\langle (\partial^2_rv)(r_0, \cdot), v(r_0, \cdot) \rangle - h^2(n-1)r_0^{n-2}\langle (\partial_rv)(r_0, \cdot), v(r_0, \cdot) \rangle.
\end{split}
\end{equation}
We can express the differential operators $\partial_r$ and $\partial^2_r$ with respect to the Euclidean coordinate system,
\begin{equation} \label{r derivatives}
\partial_r = r^{-1} \sum_{j=1}^n x_j \partial_{x_j}, \qquad \partial^2_r = r^{-2} \sum_{k=1}^n x_k \sum_{j=1}^n x_j \partial_{x_k}\partial_{x_j}.
\end{equation}
Thus by \eqref{for self-adjointness}, \eqref{rewrite laplace beltrami} and \eqref{r derivatives}, all terms in \eqref{Fr0} tend to zero as $r_0 \to 0$, except for possibly $\|hu'(r, \cdot) \|^2$ in dimension three, which in that case tends to $|v(0)|^2\int_{\US^{n-1}} d\theta$. We conclude
\begin{equation*}
\lim_{r_0 \to 0} w(r_0) F(r_0) = w(0) F(0) = \begin{cases} \omega_{n-1} w(0)|v(0)|^2  & n = 3, \\
0 & n \ge 4,\end{cases}
\end{equation*}
where $\omega_{n-1}$ is the $(n-1)$-dimensional volume of $\US^{n-1}$.

Thus in view of \eqref{pre est after integrating} and $0 < w \le 1$,
\begin{equation} \label{est after integrating}
\begin{split}
\int_{0}^\infty  Ew' \| u\|^2 &+ w' \| hu'\|^2 dr \\
&\le 2 \big( \int_{0}^\infty  \frac{1}{h^2w'} \| P^{\pm}(h) u \|^2 dr \big)^{1/2} \big(\int_{0}^\infty  w'\| hu'\|^2 dr \big)^{1/2} \\
&+  \frac{2\ep}{h} (\int_0^\infty   \|u\|^2 dr)^{1/2} (\int_0^\infty \|hu'\|^2 dr)^{1/2}. 
\end{split}
\end{equation}
We now estimate,
\begin{equation*}
    \begin{split}
        \int_0^\infty \|hu'\|^2dr &= \real \int_0^\infty \langle u, -h^2u'' \rangle dr \\
        &= \real \big( \int_{0}^\infty \langle u, P^\pm(h)u \rangle dr + \int_0^\infty \langle u, (E - V  - h^2 r^{-2}\Lambda)u \rangle dr \mp i \ep \int^\infty_0 \|u\|^2 dr \big)\\
        &= \real \int_{0}^\infty \langle u, P^\pm(h)u \rangle dr + \int_0^\infty \langle u, (E - V  - h^2 r^{-2}\Lambda)u \rangle dr \\
        &\le \big( \int_0^\infty \frac{1}{w'} \|P^\pm(h)u\|^2dr \big)^{1/2} \big(\int_0^\infty w'\|u\|^2dr \big)^{1/2} + (E+\|V^-\|_{L^\infty}) \int_0^\infty \|u\|^2dr,
    \end{split}
\end{equation*}
and 
\begin{equation*} 
    \begin{split}
        \ep \int_0^\infty \|u\|^2 dr &= \ep \| v\|^2_{L^2} \\
        &= |\imag \langle (P(h) - E \pm i\ep)v, v \rangle_{L^2}| \\
        &= \big| \imag \int_0^\infty \langle P^{\pm}(h)u, u \rangle dr \big| \\
        &\le \big( \int_0^\infty \frac{1}{w'} \|P^\pm(h)u\|^2dr \big)^{1/2} \big(\int_0^\infty w'\|u\|^2dr \big)^{1/2}.
    \end{split}
\end{equation*}
Combining these gives 

\begin{equation*}
\frac{\ep^2}{h^2} \int_0^\infty \|u\|^2 dr \cdot \int_0^\infty \|hu'\|^2dr  \le (E + \ep+ \|V^-\|_{L^\infty}) \int_0^\infty \frac{1}{h^2w'} \|P^\pm(h)u\|^2dr \cdot \int_0^\infty w'\|u\|^2dr.
\end{equation*}
Plugging this into \eqref{est after integrating} yields

\begin{equation} \label{est after integrating 2}
\begin{split}
\int_{0}^\infty Ew' \| u\|^2 &+  w'\| hu'\|^2 dr \\
&\le  2 \big( \int_0^\infty \frac{1}{h^2w'} \|P^\pm(h)u\|^2dr \big)^{1/2}\\
&\cdot \big( \big(\int_0^\infty w'\|hu'\|^2dr \big)^{1/2} + (E+\ep+\|V^-\|_{L^\infty})^{1/2} \big(\int_0^\infty w'\|u\|^2dr \big)^{1/2} \big).
\end{split}
\end{equation}

Now restrict $E > 0$ and complete the square in \eqref{est after integrating 2} to find,

\begin{equation} \label{complete the square}
\begin{split}
\big( E^{1/2} \big( \int_0^\infty& w'\|u\|^2dr \big)^{1/2} - \frac{(E+\ep+\|V^-\|_{L^\infty})^{1/2} }{E^{1/2}} \big( \int_0^\infty \frac{1}{h^2w'} \|P^\pm(h)u\|^2dr 
\big)^{1/2} \big)^2 \\
&+ \big( \big( \int_0^\infty w'\|hu'\|^2dr \big)^{1/2} - \big( \int_0^\infty \frac{1}{h^2w'} \|P^\pm(h)u\|^2dr 
\big)^{1/2} \big)^2\\
&\le \frac{2E + \ep+\|V^-\|_{L^\infty}}{E} \int_0^\infty \frac{1}{h^2w'} \|P^\pm(h)u\|^2dr.
\end{split}
\end{equation}
Dropping the second term on the left side of \eqref{complete the square} implies, for all $E > 0$ and $\ep \ge 0$,
\begin{equation} \label{E and ep est}
\begin{split}
 E^{1/2} \big(& \int_0^\infty  w'\|u\|^2dr \big)^{1/2} \\
 &\le \big(\frac{(E+\ep+\|V^-\|_{L^\infty})^{1/2}}{E^{1/2}} + \frac{(2E+\ep+\|V^-\|_{L^\infty})^{1/2}}{E^{1/2}} \big) \big( \int_0^\infty \frac{1}{h^2w'} \|P^\pm(h)u\|^2dr \big)^{1/2}.
 \end{split}
\end{equation}

Recall that $w' = C_V \delta (C_V + \delta)^{-1} (r + 1)^{-1-\delta}$ and $\delta = 2s -1$. From \eqref{E and ep est} and the density argument in Appendix \ref{density appendix} we get, rewriting $z=E \pm i\varepsilon$, that
\begin{equation}\label{thm1.1_2}
\|(1+r)^{-2s}(P-z)^{-1}(1+r)^{-2s}\|_{L^2(\mathbb{R}^n)\to L^2(\mathbb{R}^n)}\leq \frac{\mathfrak{C}(E,s,\varepsilon,\|V^-\|_{L^\infty})}{h}\,
\end{equation}
where 
\begin{equation}\label{thm1.1aux}
\mathfrak{C}(E,s,\varepsilon, \|V^-\|_{L^\infty})=\dfrac{C_V+\delta}{E\delta C_V}\left((E+\ep+\|V^-\|_{L^\infty})^{1/2}+(2E+\ep+\|V^-\|_{L^\infty})^{1/2}\right).
\end{equation}

From this point on, we assume $V^-=0$. Consider the sector $\{ z \in \C : |\imag z| < \alpha \real z \}$ for $0 < \alpha < 1$. From \eqref{E and ep est}, we get for all $h > 0$,$E \pm i \ep \in \{ z \in \C : |\imag z| < \alpha \real z \}$, and $u \in r^{n-1/2}C^\infty_0(\R^n)$, 

\begin{equation} \label{alpha est}
\begin{split}
 |(E &\pm i \ep)^{1/2}|\big( \int_0^\infty  (r +1)^{-2s}\|u\|^2dr \big)^{1/2} \\ &\le h^{-1} (1 + \alpha^2)^{1/4} \big( \frac{1}{\delta} + \frac{1}{C_V} \big)((1 + \alpha)^{1/2} + (2 + \alpha)^{1/2}) \big( \int_0^\infty (r +1)^{2s} \|P^\pm(h)u\|^2dr \big)^{1/2},
  \end{split}
\end{equation}
Here, our branch of the complex square root is chosen so that $\imag (E \pm i\ep)^{1/2} > 0$, and we used that $|(E \pm i\ep)|^{1/2} = (E^2 + \ep^2)^{1/4} \le E^{1/2}(1 + \alpha^2)^{1/4}$ for $E \pm i \ep \in \{ z \in \C : |\imag z| < \alpha \real z \}$. 
Since $u \in r^{(n-1)/2} C^\infty_0(\R^n)$, a standard density argument, which we review in Appendix \ref{density appendix}, shows that \eqref{alpha est} implies, 
\begin{equation} \label{resolv est in sector}
\begin{split}
\|z^{1/2}(r + 1)^{-s}(P(h) &-z)^{-1}(r +1)^{-s}\|_{L^2 \to L^2} \\
&\le h^{-1}(1 + \alpha^2)^{1/4} \big( \frac{1}{\delta} + \frac{1}{C_V} \big)((1 +\alpha)^{1/2} + (2 + \alpha)^{1/2}),
\end{split}
\end{equation}
on $\{ z \in \C : |\imag z| < \alpha \real z \}$ and for any $0 < \alpha < 1$. To extend this bound to all $z \in \C \setminus [0, \infty)$, we use the Phragm\'en-Lindel\"of principle \cite{em} in the following way. For $u, v \in L^2(\R^n)$, put 
\begin{equation*}
U(z) \defeq z^{1/2} \langle (r + 1)^{-s}(P(h) -z)^{-1}(r +1)^{-s}u, v \rangle_{L^2}.
\end{equation*}
Then $U(z)$ in analytic in $\Omega_\alpha \defeq \{ z \in \C : \alpha \real z < |\imag z|\}$. By \eqref{resolv est in sector}, on $\partial \Omega_\alpha \setminus \{0 \}$ we have
\begin{equation} \label{U on bdry}
|U(z)| \le h^{-1}(1 + \alpha^2)^{1/4} \big( \frac{1}{\delta} + \frac{1}{C_V} \big) ((1 + \alpha)^{1/2} + (2 +\alpha)^{1/2}) \|u\|_{L^2} \|v\|_{L^2}. 
\end{equation}
On the other hand, in $\Omega_\alpha$, we have the standard bound
\begin{equation} \label{std bd}
|U(z)| \le \frac{|z|^{1/2}\|u \|_{L^2} \|v\|_{L^2}}{\dist(z,[0, \infty))} = \begin{cases} \frac{\|u \|_{L^2} \|v\|_{L^2}}{|z|^{1/2}} & \real z < 0, \\ 
\frac{|z|^{1/2}\|u \|_{L^2} \|v\|_{L^2}}{|\imag z|} & \real z \ge 0, \, z \in \Omega_\alpha,
\end{cases}
\end{equation}
where we used
\begin{equation*}
\frac{1}{\dist(z,[0, \infty))} = \frac{1}{ \inf_{r \ge 0} ((\real z - r)^2 + (\imag z)^2)^{1/2}} = \begin{cases} \frac{1}{|z|} & \real z <0 , \\ 
\frac{1}{|\imag z|} & \real z \ge 0, \, z \in \Omega_\alpha. 
\end{cases}
\end{equation*}

Finally, define, $g(z) = e^{i (z^{-1})^{1/2}}$, where our branch of the square root is as above. In $\Omega_\alpha$, $|g(z)| \le e^{-c_\alpha |z|^{-1/2}}$ for some $0 < c_\alpha < 1$ depending on $\alpha$. It follows because, from the definition of $\Omega_\alpha$,  there exists $\theta_\alpha \in (0, \pi/4)$ so that any $z \in \Omega_\alpha$ takes the form $|z|e^{i\theta}$ with $\theta_\alpha < \theta < 2\pi -\theta_\alpha$. Whence $\real (i (z^{-1})^{1/2}) = -|z|^{-1/2}\sin(\theta/2) \le -|z|^{-1/2}\sin(\theta_\alpha/2)$. Combining with \eqref{std bd} gives
\begin{equation} \label{sigma bd}
\limsup_{z \to 0, \, z \in \Omega_\alpha} |g(z)|^{\sigma} |U(z)| = 0, \qquad \sigma > 0.
\end{equation}
Therefore, from \eqref{U on bdry} and \eqref{sigma bd}, the Phragm\'en Lindel\"of Theorem (Theorem \ref{pl thm} in Appendix \ref{pl appendix}) implies that \eqref{resolv est in sector} holds for all $z \in \Omega_\alpha$ too. Sending $\alpha \to 0^+$ completes the proof of \eqref{nontrap est repulsive V}.

To prove \eqref{low freq est repulsive V}, start again at \eqref{complete the square} and drop the first term on the left hand side. Still working on $\{ z \in \C : |\imag z| < \alpha \real z \}$, some manipulations give
\begin{equation}\label{uprimebound}
\left(\int_0^\infty w'\|hu'\|^2\,dr\right)^{1/2}\leq (1+\sqrt{2+\alpha})\left(\int_0^\infty \frac{1}{h^2 w'} \|P^{\pm}(h)u\|^2 dr\right)^{1/2}.
\end{equation}
By integration by parts,
\begin{equation*}
\begin{split}
\int_0^\infty (r +1)^{-3-\delta}\|u\|^2\,dr&=\frac{2}{2+\delta}\int_0^\infty (r +1)^{-2-\delta}\real\langle u,u'\rangle\,dr\\
&\leq h^{-1}\left(\int_0^\infty(r +1)^{-1-\delta}\|hu'\|^2\,dr\right)^{1/2}\left(\int_0^\infty (r +1)^{-3-\delta}\|u\|^2\,dr\right)^{1/2},
\end{split}
\end{equation*}
which implies
\begin{equation}\label{poincare}
\left(\int_0^\infty (r + 1)^{-3-\delta}\|u\|^2\,dr\right)^{1/2}\leq h^{-1}\left(\int_0^\infty(r +1)^{-1-\delta}\|hu'\|^2\,dr\right)^{1/2}.
\end{equation}

From \eqref{uprimebound}, \eqref{poincare} and $w' = C_V \delta (C_V + \delta)^{-1} (r + 1)^{-1-\delta}$, 
\begin{equation*}
\left(\int_0^\infty (r+1)^{-3-\delta}\|u\|^2\,dr\right)^{1/2}\leq h^{-2}\big( \frac{1}{\delta} + \frac{1}{C_V} \big) (1+\sqrt{2+\alpha})  \left(\int_0^\infty (r +1)^{1 + \delta}\|P^{\pm}(h)u\|^2\,dr\right)^{1/2}.
\end{equation*}
Using again the density argument in Appendix \ref{density appendix}, for $0 < \delta < 1$,
\begin{equation}\label{resolv est in sector low freq}
\|(1+r)^{-\frac{3+\delta}{2}}(P(h)-z)^{-1}(1+r)^{-\frac{1+\delta}{2}} \|_{L^2 \to L^2 }\leq h^{-2}\left(\delta^{-1}+C_V^{-1}\right)  (1+\sqrt{2+\alpha}),
\end{equation}
on $\{ z \in \C : |\imag z| < \alpha \real z \}$. Then, as above, \eqref{std bd}, the Phragm\'en Lindel\"of Theorem, and sending $\alpha \to 0^+$, imply
\begin{equation*}
\|(1+r)^{-\frac{3+\delta}{2}}(P(h)-z)^{-1}(1+r)^{-\frac{1+\delta}{2}} \|_{L^2 \to L^2 }\leq h^{-2}\left(\delta^{-1}+C_V^{-1}\right)  (1+\sqrt{2}), \qquad z \in \C \setminus [0, \infty).
\end{equation*}
Since the norm of an operator and its adjoint coincide,
\begin{equation*}
\|(1+r)^{-\frac{1+\delta}{2}} (P(h)-z)^{-1} (1+r)^{-\frac{3+\delta}{2}} \|_{L^2 \to L^2 }\leq h^{-2}\left(\delta^{-1}+C_V^{-1}\right)  (1+\sqrt{2}), \qquad z \in \C \setminus [0, \infty).
\end{equation*}
The three lines lemma then says that for fixed $z \in \C \setminus [0, \infty)$, the analytic mapping
\begin{equation*}
\lambda \mapsto (1+r)^{-\frac{3+\delta}{2} +\lambda }(P(h)-z)^{-1}(1+r)^{-\frac{1+\delta}{2} - \lambda}, \qquad 0 < \real \lambda < 1, 
\end{equation*}
(with values in the space of bounded operators $L^2(\R^n) \to L^2(\R^n)$) obeys 
\begin{equation} \label{rest}
\| (1+r)^{-\frac{3+\delta}{2} + \theta }(P(h)-z)^{-1}(1+r)^{-\frac{1+\delta}{2} - \theta} \|_{L^2 \to L^2} \le  h^{-2}\left(\delta^{-1}+C_V^{-1}\right) (1+\sqrt{2}), \qquad \theta \in [0, 1]. 
\end{equation}

 Having established \eqref{rest}, to finish, we need to see that we can choose $\delta$ and $\theta$ appropriately to arrive at \eqref{low freq est repulsive V}. That is, we need to attain the more general weights characterized by $s_1, s_2 > 1/2$, $s_1 + s_2 > 2$. However, because we have the restrictions $\delta \in (0,1)$ and $\theta \in [0,1]$, we first need to make reductions as follows. Since decreasing $s_1$ or $s_2$ in \eqref{low freq est repulsive V} increases the left side, it suffices to suppose $s_1, \, s_2 > 1/2$, $2< s_1 + s_2 < 3$. Furthermore, by taking the adjoint, it is no restriction to have $s_1 \le s_2$. {If we write $s_1 = (1 + 2\delta_2)/2$ for some $\delta_2 > 0$, then we may replace $s_2$ by $\min(s_2, (3 + \delta_2)/2 )$. Having made these reductions, \eqref{low freq est repulsive V} follows from \eqref{rest} by setting $\delta = s_1 + s_2 -2 < 1$, $\theta = (s_2 - s_1 + 1)/2 \le (4 - \delta_2)/4 < 1$. } 
 
 \begin{remark} \label{no zero eval remark}
 The bound \eqref{E and ep est} with $\ep = 0$ rules out $P(h)$ having an eigenvalue $E > 0$. When $V^- = 0$, a zero eigenvalue is ruled out by combining \eqref{est after integrating 2} (with $E = \ep = 0$) with \eqref{poincare}.
 \end{remark}

\section{Resolvent bounds for wave decay} \label{Helmholtz resolvent section}

In this section, we consider the operator $P \defeq P(1) = -\Delta + V$, with $P(h)$ as in \eqref{P}; $V$ obeys \eqref{nonneg} through \eqref{AC loc} and \eqref{V prime cond}. As a consequence of Theorem \ref{nontrap est thm repulsive V}, we prove several more resolvent bounds for $P$, which enable us in Section \ref{wave decay section} to establish weighted energy decay for the solution to the wave equation \eqref{wave eqn srp}. Throughout this section, $C$ denotes a positive constant whose precise value may change, but is always independent $\lambda$, which plays the role of our spectral parameter.

\begin{lemma} \label{lap lemma}
Fix $s_1,\, s_2 > 1/2$ with $s_1 + s_2 > 2$. There exist $C > 0$ so that for all $\lambda \in \C$ with $0 < |\imag \lambda| \le 1$, and for all  multiindices $\alpha_1, \, \alpha_2$ with $|\alpha_1| + |\alpha_2| \le 2$,
\begin{equation} \label{lap A}
\|  \langle x \rangle^{-s_1} \partial^{\alpha_1}_x (P- \lambda^2)^{-1} \partial^{\alpha_2}_x  \langle x \rangle^{-s_2} \|_{L^2(\R^n) \to L^2(\R^n) } \le C(1 + |\real \lambda|)^{|\alpha_1| + |\alpha_2|-1 }. 
\end{equation}
\end{lemma}

\begin{proof}
Since $((P - \lambda^2)^{-1})^* = (P- (\overline{\lambda})^2)^{-1}$, to prove \eqref{lap A} is suffices to assume $\imag \lambda > 0$. 
 
First, we treat the case $|\alpha_2| = 0$. Using \eqref{nontrap est repulsive V} if $|\real \lambda|>1$ or \eqref{low freq est repulsive V} if $|\real \lambda|\leq 1$, we get
\begin{equation} \label{L2 to L2 bd Helmholtz resolvent}
  \|\langle x \rangle^{-s_1} (P- \lambda^2)^{-1}  \langle x \rangle^{-s_2} \|_{L^2 \to L^2 } \le C(1 + |\real \lambda|)^{-1}, \qquad 0 < \imag \lambda \le 1, 
\end{equation}
Recall from standard elliptic theory that for all $f \in H^2(\R^n)$ and all $\gamma > 0$,
\begin{equation} \label{std elliptic thry}
\begin{gathered}
\| f\|_{H^2} \le C( \| f\|_{L^2} +  \| \Delta f\|_{L^2}), \\
\| f\|^2_{H^1} \le C \| f\|_{L^2} \|f \|_{H^2} \le C( \gamma^{-1} \| f\|^2_{L^2} + \gamma \| \Delta f\|^2_{L^2}).
\end{gathered}
\end{equation}
Therefore, for any $f \in L^2(\R^n)$,
\begin{equation*}
\begin{split}
\| \langle x \rangle^{-s_1}& (P- \lambda^2)^{-1} \langle x \rangle^{-s_2}f\|_{H^2(\R^n)} \\
&\le C ( \| \langle x \rangle^{-s_1} (P - \lambda^2)^{-1} \langle x \rangle^{-s_2}f\|_{L^2(\R^n)} + \| (-\Delta) \langle x \rangle^{-s_1} (P - \lambda^2)^{-1} \langle x \rangle^{-s_2}f \|_{L^2(\R^n)}) \\
&\le C ( \| \langle x \rangle^{-s_1} (P- \lambda^2)^{-1} \langle x \rangle^{-s_2}f\|_{H^1(\R^n)} + \| \langle x \rangle^{-s_1} (-\Delta) (P- \lambda^2)^{-1} \langle x \rangle^{-s_2}f \|_{L^2(\R^n)}) \\
&\le C(\gamma^{-1} + |\real \lambda|^2) \| \langle x \rangle^{-s_1} (P- \lambda^2)^{-1} \langle x \rangle^{-s_2}f\|_{L^2(\R^n)} \\
&+ C\gamma \| \Delta \langle x \rangle^{-s_1} (P - \lambda^2)^{-1} \langle x \rangle^{-s_2}f\|_{L^2(\R^n)}\\
 &+ C \| f \|_{L^2(\R^n)}. 
\end{split}
\end{equation*}
Selecting $\gamma$ sufficiently small depending on $C$, and applying \eqref{L2 to L2 bd Helmholtz resolvent} yields 
\begin{equation} \label{L2 to H2 bd Helmholtz resolvent}
\| \langle x \rangle^{-s_1}(P- \lambda^2)^{-1}  \langle x \rangle^{-s_2}f\|_{H^2(\R^n)} \le C(1 + |\real \lambda|) \| f \|_{L^2(\R^n)},
\end{equation}
as desired. This confirms \eqref{lap A} for $|\alpha_1| = 2$. For $|\alpha_1| = 1$ (still with $|\alpha_2| = 0$), combine \eqref{L2 to L2 bd Helmholtz resolvent} and \eqref{L2 to H2 bd Helmholtz resolvent} via the second line of \eqref{std elliptic thry}.

If $|\alpha_2| > 0$, let $f \in  C^{\infty}_0(\R^n)$, and put $u = \langle x \rangle^{-s_1} (P - \lambda^2)^{-1} \langle x \rangle^{-s_2} \partial^{\alpha_2}_x f$. We need to show
\begin{equation*}
\| u\|_{H^{|\alpha_1|}} \le C(1 + |\real \lambda|)^{|\alpha_1| + |\alpha_2| -1} \| f\|_{L^2}, \qquad H^0 \defeq L^2(\R^n).   
\end{equation*} 
If $|\alpha_1| = 0$, we use self-adjointness and that we have already showed  $ \| \langle x \rangle^{-s_2} (P - \lambda^2)^{-1} \langle x \rangle^{-s_1} f \|_{H^j} \le C (1 + |\real \lambda|)^{j-1} \| f \|_{L^2}$, $j \in \{0,1,2\}$, to get
\begin{equation*}
\begin{split}
\| u\|^2_{L^2} &= \langle u, \langle x \rangle^{-s_1} (P - \lambda^2)^{-1} \langle x \rangle^{-s_2} \partial^{\alpha_2}_x f \rangle_{L^2} \\
&\le \| \partial^{\alpha_2}_x \langle x \rangle^{-s_2} (P - (\overline{\lambda})^2)^{-1} \langle x \rangle^{-s_1} u \|_{L^2} \| f\|_{L^2} \\
&\le C(1 + |\real \lambda|)^{|\alpha_2|  -1} \| u \|_{L^2} \| f\|_{L^2}.
\end{split}
\end{equation*}
If $|\alpha_1| = 1$, we recognize that $(P - \lambda^2)u = \langle x \rangle^{-s_1 - s_2} \partial^{\alpha_2}_x f + [-\Delta, \langle x \rangle^{-s_1}] \langle x \rangle^{s_1} u$. Then multiply by $\overline{u}$, integrate over $\R^n$, and integrate by parts as appropriate
\begin{equation*}
\begin{split}
\| \nabla u \|_{L^2}^2 &= \int (\lambda^2 - V) |u|^2 - \int \partial^{\alpha_2}_x (\langle x \rangle^{-s_1 - s_2} \overline{u}) f + \int \overline{u} [-\Delta, \langle x \rangle^{-s_1}] \langle x \rangle^{s_1} u \\
& \le \lambda^2 \int |u|^2 - \int \partial^{\alpha_2}_x (\langle x \rangle^{-s_1 - s_2} \overline{u}) f + \int \overline{u} [-\Delta, \langle x \rangle^{-s_1}] \langle x \rangle^{s_1} u.
\end{split}
\end{equation*}
Because both
\begin{equation*}
[-\Delta, \langle x \rangle^{-s_1}] \langle x \rangle^{s_1} = (-\Delta \langle x \rangle^{-s_1})  \langle x \rangle^{s_1} - 2 (\nabla \langle x \rangle^{-s_1}) \cdot \nabla \langle x \rangle^{s_1}, 
\end{equation*} 
and $\partial^{\alpha_2}_x \langle x \rangle^{-s_1 - s_2}$ are first order differential operators with bounded coefficients, we conclude, for all $\gamma > 0$,
\begin{equation*}
\begin{split}
\| \nabla u \|^2_{L^2} &\le C_\gamma ( (1 + |\real \lambda|)^2 \| u\|^2_{L^2} + \| f\|^2_{L^2} ) + \gamma \| \nabla u \|^2_{L^2} \\
&\le C_\gamma (1 + |\real \lambda|)^{2} \| f\|^2_{L^2} + \gamma \| \nabla u \|^2_{L^2},
\end{split}
\end{equation*}
for some $C_\gamma > 0$ depending on $\gamma$. 

Fixing $\gamma$ small enough, we absorb the second term on the right side into the left side, confirming \eqref{lap A} when $|\alpha_1| = |\alpha_2| = 1$.\\
\end{proof}

Next, we prove an estimate for the $\lambda$-derivative of the weighted resolvent, which requires the extra short range conditions \eqref{stronger bd V} and \eqref{delta restrictions} on the potential. As input we need the following bound for the weighted square of the free resolvent, which we prove in Appendix \ref{deriv free resolv appendix}.

\begin{lemma} \label{lap free resolv square lem}
Let $n \ge 3$. Suppose $s$ satisfies \eqref{s restrictions}. There exists $C > 0$ so that for all $\lambda \in \C$ with $0 < |\imag \lambda| \le 1$, and any  multiindex such that $|\alpha| \le 1$,
\begin{equation} \label{lap free resolv square}
\| \lambda \langle x \rangle^{-s} \partial^\alpha_x (-\Delta - \lambda^2)^{-2} \langle x \rangle^{-s}\|_{L^2(\R^n) \to L^2(\R^n)} \le C(1 + |\lambda|)^{|\alpha| -1}.
\end{equation}
\end{lemma} 

\begin{remark}
In \cite{vo04}, the estimate \eqref{lap free resolv square} is stated to hold in any dimension $n \ge 3$ provided $s > 3/2$. However, our proof of Lemma \ref{lap free resolv square lem} in dimension $n \ge 4$ needs $s$ larger if \eqref{lap free resolv square} is to hold uniformly as $|\lambda| \to 0$. In our approach, we use the integral kernel of $\lambda \langle x \rangle^{-s} (-\Delta - \lambda^2)^{-2} \langle x \rangle^{-s}$ to assess $L^2$-boundedness as $|\lambda| \to 0$. The kernel is given in terms of the Macdonald function \cite[10.27.4, 10.27.5]{dlmf} of order $(n/2) - 2$, along with other factors. We are able to conclude boundedness on $L^2(\R^n)$ for $s$ as in \eqref{s restrictions}.
\end{remark}

\begin{lemma} \label{resolv square lemma}
Let $n \ge 3$ and suppose $s$ is as in \eqref{s restrictions}. Assume $V$ obeys \eqref{nonneg} through \eqref{AC loc} and \eqref{V prime cond}, as well as \eqref{stronger bd V} and \eqref{delta restrictions}. There exists $C > 0$  so that for all $\lambda \in \C$ with  $0 < | \imag \lambda | \le 1$, and any $j \in \{0, 1\}$ and multiindex $\alpha$ such that $j + |\alpha| \le 1$,
\begin{equation} \label{lap square}
\| \tfrac{d}{d \lambda}\langle x \rangle^{-s}  \lambda^{j} \partial^\alpha_{x}   (P - \lambda^2)^{-1}   \langle x \rangle^{-s} \|_{L^2(\R^n) \to L^2(\R^n)} \le C.
\end{equation}
\end{lemma}

\begin{proof}
Without loss of generality, we take $s$ sufficiently close to (but larger than) $(n + 3)/4$ when $n \neq 8$, or sufficiently close to (but larger than) $3$ when $n = 8$, so that by \eqref{delta restrictions} we may fix $s' > 1/2$ so that $s + s' < \delta$. We note also that with these choices we have $s' + s > 2$, so we are permitted to apply \eqref{lap A} as the need arises.

We begin from the resolvent identity 
\begin{equation} \label{resolv id for deriv}
(P - \lambda^2)^{-1} \langle x \rangle^{-s} (I + K(\lambda)) = R_0(\lambda) \langle x \rangle^{-s}, 
\end{equation}
where $K(\lambda) \defeq V(x) \langle x \rangle^{s + s'} \langle x \rangle^{-s'} R_0(\lambda) \langle x \rangle^{-s}$ and $R_0(\lambda) \defeq (-\Delta - \lambda^2)^{-1}$.

It is well known that $\langle x \rangle^{-s'} R_0(\lambda) \langle x \rangle^{-s} : L^2(\R^n) \to H^2(\R^n)$ has a continuous extension from either half-plane ($\pm \imag \lambda > 0$) to $\R$ \cite[Proposition 2.4]{gimo74}. Let us denote this extension by $R^{\pm}_{0, s',s}(\lambda)$ and put $K^\pm(\lambda) = V(x) \langle x \rangle^{s + s'} R^{\pm}_{0, s',s}(\lambda)$ 

We now show that $K^\pm(\lambda)$ is a compact operator $L^2(\R^n) \to L^2(\R^n)$. To see this, observe that we may write $K^\pm(\lambda)$ as the sum
\begin{equation*}
K^\pm(\lambda) =  (\chi \langle x \rangle^{s + s'} V) R^{\pm}_{0, s',s}(\lambda)  + ((1 - \chi)V \langle x \rangle^{\delta}) \langle x \rangle^{s + s'- \delta}  R^{\pm}_{0, s',s}(\lambda).
\end{equation*}
where $\chi \in C^\infty_0(\R^n ; [0,1])$ is identically one near the origin in $\R^n$ and supported in $B(0,1)$. The second operator on the right side is compact by \cite[Theorem B.4]{dz}). The first operator on the right side is compact as follows: it may be viewed as the composition of bounded $R^{\pm}_{0, s',s}(\lambda) : L^2(\R^n) \to H^2(\R^n)$ followed by multiplication by $\chi \langle x \rangle^{s + s'} V$. Due to \eqref{for self-adjointness} and Lemma \ref{faris lemma}, we have $\|\chi \langle x \rangle^{s + s'} V u\|_{L^2(\R^n)} \le C \| u \|_{H^1(B(0,1))}$ for some $C > 0$ and all $u \in H^2(\R^n)$. By the Kondrachov embedding theorem the inclusion $H^2(B(0,1)) \to H^1(B(0,1))$ is compact. So compactness of $(\chi \langle x \rangle^{s + s'} V) R^{\pm}_{0, s',s}(\lambda)$ holds as desired. 

We claim further that $I + K^\pm(\lambda)$ is invertible $L^2(\R^n) \to L^2(\R^n)$ for all $\lambda$ with $\pm \imag \lambda \ge 0$. By compactness of $K^\pm(\lambda)$ and the Fredholm alternative \cite[Theorem VI.14]{reedsimon1}, we have that $I + K^\pm(\lambda)$ is invertible if we can show that for $g \in L^2(\R^n)$, $(I + K^\pm(\lambda))g = 0$ implies $g = 0$. To this end, put $u \defeq \langle x \rangle^{s'} R^{\pm}_{0, s',s}(\lambda)g$, which belongs to $\langle x \rangle^{s'} H^2(\R^n)$. If we can show $u = 0$, then in fact $g = 0$. This is because $(-\Delta - \lambda^2) u = \langle x \rangle^{-s} g$ in the distributional sense.

Now let us show $u = 0$. If $\lambda^2 \in \C \setminus [0, \infty)$ (so that $K^\pm(\lambda) = K(\lambda)$), this follows immediately from $(P - \lambda^2) u = \langle x \rangle^{-s}g + V R_0(\lambda) \langle x \rangle^{-s} g = \langle x \rangle^{-s}(I + K(\lambda))g = 0.$ If $\lambda^2 \in [0, \infty)$, the idea is the same, but we incorporate a limiting step. Set $u_{\pm, \ep} = (-\Delta - \lambda^2 \pm i \ep)^{-1} \langle x \rangle^{-s}g$. We have that \cite[Proposition 2.4]{gimo74} implies that  $\langle x \rangle^{-s'} u_{\pm, \ep}$ converges to $\langle x \rangle^{-s'}u$ in $H^2(\R^n)$ as $\ep \to 0^+$. We also have that 
\begin{equation*}
\begin{split}
u_{\pm, \ep} &= (-\Delta - \lambda^2 \pm i \ep)^{-1} \langle x \rangle^{-s}g \\
&=  (P - \lambda^2 \pm i\ep)^{-1} (P - \lambda^2 \pm i\ep) (-\Delta - \lambda^2 \pm i \ep)^{-1} \langle x \rangle^{-s}g \\
&= (P - \lambda^2 \pm i\ep)^{-1} \langle x \rangle^{-s} (I + V \langle x \rangle^{s}(-\Delta - \lambda^2 \pm i \ep)^{-1} \langle x \rangle^{-s}) g.
\end{split}
\end{equation*} 

Therefore, by \eqref{low freq est repulsive V},
\begin{equation*}
\begin{split}
\| \langle x \rangle^{-s'}u \|_{L^2} &= \lim_{\ep \to 0^+} \| \langle x \rangle^{-s'} u_{\pm, \ep} \|_{L^2} \\
&\le C \lim_{\ep \to 0^+} \| (I + V \langle x \rangle^{s}(-\Delta - \lambda^2 \pm i \ep)^{-1} \langle x \rangle^{-s}) g \|_{L^2} \\
&= \| (I + K^\pm(\lambda)) g \|_{L^2} = 0.
\end{split} 
\end{equation*}    
Thus we have demonstrated that $I + K^\pm(\lambda)$ is invertible for $\pm \imag \lambda \ge 0$. As $\lambda \to \infty$, $\|K(\lambda)\|_{L^2 \to L^2} \to 0$ thanks to \eqref{lap A}, hence we can compute $(I + K^\pm(\lambda))^{-1}$ by a Neumann series, thanks to \eqref{lap A}. Therefore
\begin{equation} \label{bd I plus K inv}
\|(I + K^\pm(\lambda))^{-1} \|_{L^2 \to L^2} \le C.
\end{equation}

Now for $0 < |\imag \lambda| \le 1$ take the $\lambda$-derivative of \eqref{resolv id for deriv},

\begin{equation} \label{deriv resolv id}
\begin{split}
\big( \tfrac{d}{d \lambda}\langle x \rangle^{-s} & \lambda^{j} \partial^\alpha_x   (P - \lambda^2)^{-1}   \langle x \rangle^{-s} \big) (I + K(\lambda)) \\
&=\tfrac{d}{d \lambda}\langle x \rangle^{-s} \lambda^{j} \partial^\alpha_x  R_0(\lambda )  \langle x \rangle^{-s}  \\
&-2 \langle x \rangle^{-s} \lambda^{j} \partial^\alpha_x   (P - \lambda^2)^{-1}   \langle x \rangle^{-s'} V \langle x \rangle^{s + s'} \lambda \langle x \rangle^{-s} (-\Delta - \lambda^2)^{-2} \langle x \rangle^{-s}, 
\end{split}
\end{equation}
where we used 
\begin{equation} \label{deriv of resolv}
\tfrac{d}{d \lambda}\langle x \rangle^{-s'}  \partial^\alpha_x R_0(\lambda )  \langle x \rangle^{-s} = 2 \lambda \langle x \rangle^{-s'} \partial^\alpha_x (-\Delta - \lambda^2)^{-2}  \langle x \rangle^{-s}, \qquad j \in \{0,1\}.
\end{equation}
The operator norm $L^2(\R^n) \to L^2(\R^n)$ of the term in the second line of  \eqref{deriv resolv id} is bounded above by a constant due to \eqref{lap free resolv square} and \eqref{deriv of resolv}. As for the third line, $\|\langle x \rangle^{-s} \lambda^{j} \partial^\alpha_x  (P - \lambda^2)^{-1}   \langle x \rangle^{-s'}\|_{L^2 \to L^2} \le C$ by \eqref{lap A}. Moreover
\begin{equation*}
\begin{split}
\| V& \langle x \rangle^{s + s'} \lambda \langle x \rangle^{-s} (-\Delta - \lambda^2)^{-2} \langle x \rangle^{-s} \|_{L^2 \to L^2} \\
& \le C \| \lambda \langle x \rangle^{-s} (-\Delta - \lambda^2)^{-2} \langle x \rangle^{-s} \|_{H^1\to L^2} 
\end{split}
\end{equation*}
since multiplication by $V \langle x \rangle^{s + s'}$ is a bounded operator $H^1(\R^n) \to L^2(\R^n)$ (see \eqref{for self-adjointness} and Lemma \ref{faris lemma}). Finally, because $\| \lambda \langle x \rangle^{-s} (-\Delta - \lambda^2)^{-2} \langle x \rangle^{-s} \|_{H^1\to L^2} \le C$ by \eqref{lap free resolv square}, the proof of \eqref{lap square} is complete.\\
\end{proof}

\section{Proof of Theorem \ref{wave decay thm}} \label{wave decay section}

In this section we prove Theorem \ref{wave decay thm} by combining the resolvent bounds of the previous section with an argument appearing in \cite[Section 3]{vo04}. As before we use the notation $P = -\Delta + V : L^2(\R^n) \to L^2(\R^n)$, $n \ge 3$, where $V$ obeys \eqref{nonneg} through \eqref{V prime cond} along with \eqref{stronger bd V} and \eqref{delta restrictions}. 

In several steps below, we use that for all $0 \le \alpha \le 1$, there exists $C > 0$ so that for any $f \in H^1(\R^n)$, 
\begin{equation} \label{Poincare etc}
\| V^\alpha f \|^2_{L^2} \le C (\| \nabla f \|^2_{L^2}  + \|f\|^2_{L^2}) \le C\| \nabla f \|^2_{L^2}.
\end{equation}
The first inequality follows from \eqref{for self-adjointness}, \eqref{stronger bd V} and Lemma \ref{faris lemma}, while the second follows from the Poincar\'e inequality (as we work in dimension $n \ge 3$). 

Given $s > 0$ and $u$ as in \eqref{soln spect thm} solving the wave equation \eqref{wave eqn srp}, with compactly supported initial conditions $u(0, x) = u_0(x) \in H^1(\R^n)$, $\partial_t u(0,x) = u_1(x) \in L^2(\R^n)$, define
\begin{equation*}
\begin{gathered}
E_s(t) \defeq \int_{\R^n} \langle x \rangle^{-2s}( |\partial_t u(t,x)|^2 + |\nabla u(t,x)|^2 + |u(t,x)|^2)dx, \\
E(0) \defeq \| \nabla u_0 \|^2_{L^2} + \| u\|^2_{L^2}. 
\end{gathered}
\end{equation*}

\begin{lemma}
If $s > 1/2$ and $V$ satisfies \eqref{nonneg} through \eqref{V prime cond}, there exists $C > 0$ so that 
\begin{equation} \label{E integrable}
\int_0^\infty E_s(\tau) d\tau \le CE(0).  
\end{equation}
If in addition $s$ satisfies \eqref{s restrictions} and $V$ \eqref{stronger bd V} and \eqref{delta restrictions}, there exists $C > 0$ so that for $t \ge 1$,
\begin{equation} \label{t minus two bd}
\int_t^\infty E_s(\tau) d\tau \le Ct^{-2} E(0).
\end{equation}
\end{lemma}
\begin{proof}
Choose $\phi \in C^\infty(\R)$, $\phi \ge 0$, $\phi(t) = 0$ near $(-\infty, 1/2]$, $\phi(t) = 1$ near $[1, \infty)$. Since $(\partial^2_t + P)u = 0$, where $P = -\Delta + V$, It holds that
\begin{equation} \label{v}
(\partial^2_t + P) \phi u = (\phi'' + 2 \phi' \partial_t)u \defeq v(t). 
\end{equation} 
Thus, by Duhamel's formula for the solution to an inhomogeneous wave equation with zero initial conditions, 
\begin{equation*}
\phi u(t) = \int^t_0 \frac{\sin(t - \tau)\sqrt{P}}{\sqrt{P}} v(\tau) d\tau.
\end{equation*}
On the other hand,
\begin{equation*}
(P - (\lambda - i\ep)^2)^{-1} = \int^\infty_0 e^{-it(\lambda - i\ep)} \frac{\sin(t \sqrt{P})}{\sqrt{P}} dt, \qquad \ep > 0.
\end{equation*}
It follows from the last two identities that the Fourier transform $\widehat{\phi u}$ of $\phi u$ satisfies
\begin{equation} \label{FT identity}
\widehat{\phi u}(\lambda - i\ep) \defeq \int_{-\infty}^\infty e^{-it(\lambda - i\ep)} \phi(t)u(\cdot, t) dt =  (P - (\lambda - i\ep)^2)^{-1} \hat{v}(\lambda - i\ep).
\end{equation}

By finite propagation speed for the wave equation, and because $v(t)$ is compactly supported in $t$, $\supp_{x} v(t)$, and thus also $\supp_{x} \hat{v}(\lambda)$, is contained in some compact subset of $\R^n$ independent of $t$. Choose $\eta \in C^\infty_0(\R^n)$ such that $\eta = 1$ near $\supp_x {v}(t)$ for all $t \in \R$. By \eqref{FT identity},
\begin{equation*}
\begin{gathered}
\langle x \rangle^{-s} \widehat{\phi u} (\lambda - i\ep) = \langle x \rangle^{-s}(P - (\lambda - i\ep)^2)^{-1} \eta \hat{v}(\lambda - i\ep), \\
\langle x \rangle^{-s} \widehat{\partial_t (\phi u)}(\lambda - i\ep) = \langle x \rangle^{-s} (\lambda - i\ep) (P - (\lambda - i\ep)^2)^{-1} \eta \hat{v}(\lambda - i\ep), \\
\langle x \rangle^{-s} \nabla \widehat{\phi u}(\lambda - i\ep) = \langle x \rangle^{-s} \nabla (P - (\lambda - i\ep)^2)^{-1} \eta \hat{v}(\lambda - i\ep).
\end{gathered}
\end{equation*}
Therefore, by \eqref{lap A}, for $s > 1/2$ and $V$ obeying \eqref{nonneg} through \eqref{V prime cond}, there is $C > 0$ independent of $\lambda$ and $\ep$, so that for all $\lambda \in \R$, $0 < \ep \le 1$, we have
 \begin{equation} \label{use resolv bds}
\begin{split}
\Big \| \frac{d^k}{d\lambda^k} & \langle x \rangle^{-s} \widehat{\partial_t (\phi u)}(\lambda - i\ep)  \Big \|_{L^2}  + \Big \| \frac{d^k}{d\lambda^k} \langle x \rangle^{-s} \nabla \widehat{\phi u}(\lambda - i\ep) \Big \|_{L^2}\\
&+ \Big \| \frac{d^k}{d\lambda^k} \langle x \rangle^{-s} \widehat{\phi u}(\lambda - i\ep) \Big \|_{L^2} \le C\| \hat{v}(\lambda - i\ep)\|_{L^2} + Ck\| \widehat{tv}(\lambda - i\ep)\|_{L^2}. 
\end{split}
\end{equation}
for $k = 0$. If in addition we suppose $s$ satisfies \eqref{s restrictions} and $V$ satisfies \eqref{stronger bd V} and \eqref{delta restrictions}, then by \eqref{lap square}, \eqref{use resolv bds} holds for $k \in \{0,1\}$. Note when $k = 1$ we used the product rule and the identity $\frac{d}{d\lambda}\hat{v}(\lambda-i\varepsilon)=-i\widehat{tv}(\lambda-i\varepsilon)$. 

Next, by \eqref{use resolv bds} and Plancherel's theorem, there exist $C_1, C_2, C_3, C > 0$ independent of $\ep$ so that
\begin{equation} \label{Plancherel}
\begin{split}
\int_{-\infty}^{\infty}& ( \| \langle x \rangle^{-s} \partial_t(\phi u)\|_{L^2}^2 +  \| \langle x \rangle^{-s} \nabla (\phi u)\|_{L^2}^2 +  \| \langle x \rangle^{-s} \phi u\|_{L^2}^2)e^{-2\ep t} dt \\
&= C_1\int_{-\infty}^{\infty}( \| \langle x \rangle^{-s} \widehat{\partial_t(\phi u)}(\lambda - i\ep)\|_{L^2}^2 + \| \langle x \rangle^{-s} \nabla \widehat{\phi u}(\lambda - i\ep)\|_{L^2}^2 + \| \langle x \rangle^{-s} \widehat{\phi u}(\lambda - i\ep)\|_{L^2}^2) d\lambda \\
&\le C_2 \int_{-\infty}^{\infty} \| \hat{v}(\lambda - i\ep)\|^2_{L^2} d\lambda = C_3 \int_{-\infty}^{\infty} \| v(t)\|^2_{L^2} e^{-2\ep t} dt \le C \sup_{t \in \R} \|v(t)\|^2.
\end{split}
\end{equation}
The last constant $C$ is independent of $\ep$ because $v(t)$ has compact support in $t$, see \eqref{v}. The proof of \eqref{E integrable} is completed by sending $\ep \to 0$ in \eqref{Plancherel} and observing
\begin{equation} \label{apply Poincare}
\begin{split}
\| v(t) \|_{L^2} &\le C(\| u_0 \|_{L^2} +  \| \sqrt{P} u_0\|_{L^2} + \|u_1\|_{L^2}) \\
& \le C( \| \nabla u_0\|_{L^2} + \|u_1\|_{L^2}) = C\sqrt{E(0)}.
\end{split}
\end{equation}
Between lines one and two of \eqref{apply Poincare}, we used that for any $f \in H^2(\R^n)$ (and thus any $f \in H^1(\R^n)$, since $H^2(\R^n)$ is dense in $H^1(\R^n)$),
\begin{equation*}
\begin{split}
 \| \sqrt{P} f \|^2_{L^2}  &=  \langle f, Pf \rangle_{L^2} = \| \nabla f \|^2_{L^2} + \| \sqrt{V} f\|^2_{L^2}  \le C \|\nabla f\|^2_{L^2},
 \end{split}
\end{equation*}
with the second inequality due to \eqref{Poincare etc}.

To prove \eqref{t minus two bd}, we again use Plancherel's theorem with \eqref{use resolv bds}, so that for all $0 < \ep \le 1$ and $T \ge 1$,
\begin{equation} \label{Plancherel again}
\begin{split}
T^2\int_{T}^{\infty}& ( \| \langle x \rangle^{-s} \partial_t(\phi u)\|_{L^2}^2 +  \| \langle x \rangle^{-s} \nabla (\phi u)\|_{L^2}^2 +  \| \langle x \rangle^{-s} \phi u\|_{L^2}^2)e^{-2\ep t} dt \\
& \le \int_{-\infty}^{\infty} ( \| \langle x \rangle^{-s} t \partial_t(\phi u)\|_{L^2}^2 +  \| \langle x \rangle^{-s} t \nabla (\phi u)\|_{L^2}^2 +  \| \langle x \rangle^{-s} t \phi u\|_{L^2}^2)e^{-2\ep t} dt \\
&= C_1\int_{-\infty}^{\infty} \big( \big\| \frac{d}{d\lambda}  \langle x \rangle^{-s} \widehat{\partial_t(\phi u)}(\lambda - i\ep) \big\|_{L^2}^2 + \big\| \frac{d}{d \lambda} \langle x \rangle^{-s} \nabla \widehat{\phi u}(\lambda - i\ep)\big\|_{L^2}^2 \\
&+ \big\| \frac{d}{d\lambda} \langle x \rangle^{-s} \widehat{\phi u}(\lambda - i\ep)\big\|_{L^2}^2\big) d\lambda \\
&\le C_2 \int_{-\infty}^{\infty} \| \hat{v}(\lambda - i\ep)\|^2_{L^2} + \|\widehat{tv}(\lambda - i \ep) \|^2_{L^2}) d\lambda \\
&= C_3 \int_{-\infty}^{\infty} (\| v(t)\|^2_{L^2} + \|tv(t) \|^2_{L^2}) e^{-2\ep t} dt \le C \sup_{t \in \R} \|v(t)\|^2 \le CE(0). 
\end{split}
\end{equation}
Once again sending $\ep \to 0^+$ concludes the proof of \eqref{t minus two bd}. \\
\end{proof}

The local energy decay \eqref{LED} follows from \eqref{t minus two bd} and 

\begin{lemma}
If $s > 0$ and $V \ge 0$ satisfies \eqref{for self-adjointness} and \eqref{for self-adjointness II}, there exists $C> 0$ so that for all $t \ge 1$, 
\begin{equation} \label{bd E by its integral}
E_s(t) \le C \int_t^\infty E_s(\tau) d\tau.
\end{equation}
\begin{proof}
The strategy is the same as that of \cite[Lemma 3.2]{vo04}. Computing $\frac{d}{dt} E_s(t)$, one finds
\begin{equation} \label{deriv of wtd E}
\begin{split}
\frac{d}{dt} E_s(t) &= -2 \real \int_{\R^n} \partial_r u(t,x) \overline{\partial_t u(t,x)} \partial_r \langle x \rangle^{-2s}dx\\
&+ 2 \real  \int_{\R^n} (-Vu(t,x) \overline{\partial_t u(t,x)} + u(t,x) \overline{\partial_t u(t,x)} )\langle x \rangle^{-2s} dx.
\end{split}
\end{equation}
By \eqref{Poincare etc},
\begin{equation*}
\begin{split}
\| V \langle x \rangle^{-s} u(t,x) \|_{L^2} &\le C \|\nabla \langle x \rangle^{-s} u(t,x) \|_{L^2}\\
 &\le C\|\langle x \rangle^{-s} \nabla  u(t,x) \|_{L^2} + C\| \langle x \rangle^{-s} u(t,x) \|_{L^2}.
\end{split}
\end{equation*}
 for $C> 0$ independent of $t$, and whose precise value may change between lines. Thus we can bound the right side of \eqref{deriv of wtd E} from above by Cauchy-Schwarz, 
\begin{equation*}
\begin{split}
\frac{d}{dt} E_s(t) &\le  C\| \langle x \rangle^{-s} \partial_r u(t,x) \|_{L^2} \| \langle x \rangle^{-s}  \partial_t u(t,x)\|_{L^2} + C\| V \langle x \rangle^{-s} u(t,x) \|_{L^2} \| \langle x \rangle^{-s}  \partial_t u(t,x)\|_{L^2} \\
&+  C\| \langle x \rangle^{-s} u(t,x) \|_{L^2} \| \langle x \rangle^{-s}  \partial_t u(t,x)\|_{L^2} 
\le C E_s(t). 
\end{split}
\end{equation*}
We then have, for all $T > t \ge 1$, 
\begin{equation} \label{penult est}
E_s(t) \le E_s(T) + C_s \int_t^T E_s(\tau) d \tau.
\end{equation}
From \eqref{E integrable}, we also have a sequence $T_j \to \infty$ so that $\lim_{T_j \to \infty} E_s(T_j) =0$. So setting $T = T_j$ in \eqref{penult est} and sending $T_j \to \infty$ completes the proof of \eqref{bd E by its integral}.\\
\end{proof}
\end{lemma}

\appendix

\section{Proof of Lemma \ref{technical lemma}} \label{technical lemma appendix}

First we check that 
\begin{equation} \label{is it L1}
C^\infty_0(0, \infty) \ni \varphi \mapsto \int_{\US^{n-1}} \int_0^\infty \varphi(r) |u(r,\theta)|^2dV(r, \theta) d\theta
\end{equation}
(part of the expression in \eqref{technical dist formula}) is well defined as a distribution on $(0, \infty)$. Indeed, by \eqref{dist deriv}, for each $\theta \in \US^{n-1}$
\begin{equation*}
\int_0^\infty \varphi(r) |u(r,\theta)|^2dV(r, \theta) = -\int_0^\infty V(r,\theta) ( |u(r, \theta)|^2 \varphi(r))'dr.
\end{equation*}
Moreover, by Fubini's theorem the expression on the right side belongs to $L^1(\US^{n-1})$. Therefore the quantity in \eqref{is it L1} is well defined as a distribution on $(0,\infty)$.

Next we demonstrate that the function $r \mapsto \int_{\US^{n-1}} V^{R}(r, \theta) |u(r, \theta)|^2 d\theta$ has locally bounded variation. Suppose $[a,b] \subseteq (0, \infty)$ and $a = r_0 < r_1 < \dots r_N = b$. We have  

\begin{equation*}
\begin{split}
\sum_{k =1}^N & \Big| \int_{\US^{n-1}} V^{R}(r_k, \theta) |u(r_k, \theta)|^2 d\theta - \int_{\US^{n-1}} V^{R}(r_{k-1}, \theta) |u(r_{k-1}, \theta)|^2 d\theta \Big| \\
&\le   \int_{\US^{n-1}}  \sum_{k =1}^N |u(r_k, \theta)|^2 |V^{R}(r_k, \theta) - V^{R}(r_{k-1}, \theta)| d\theta \\
&+ \int_{\US^{n-1}} \sum_{k=1}^N |V^{R}(r_{k-1}, \theta)| | |u(r_k, \theta)|^2 - |u(r_{k-1}, \theta)|^2| d\theta. 
\end{split}
\end{equation*} 
For each $\theta$, $r \mapsto V^R(r, \theta)$ is decreasing, so 
\begin{equation*}
\begin{split}
\int_{\US^{n-1}} \sum_{k =1}^N &|u(r_k, \theta)|^2 |V^{R}(r_k, \theta) - V^{R}(r_{k-1}, \theta)| d\theta \\ &\le \|u\|^2_{L^\infty([a,b] \times \US^{n-1})} \int_{\US^{n-1}}  \sum_{k =1}^N |V^{R}(r_k, \theta) - V^{R}(r_{k-1}, \theta)| d\theta \\
 &\le \|u\|^2_{L^\infty([a,b] \times \US^{n-1})} \int_{\US^{n-1}} ( V^R(a,\theta) - V^R(b, \theta) ) d\theta \\
 &\le \omega_{n-1} \|u\|^2_{L^\infty([a,b] \times \US^{n-1})} \|V\|_{L^\infty([a,b] \times \US^{n-1})},
 \end{split}
\end{equation*}
where $\omega_n$ is the $(n-1)$-dimensional volume of $\US^{n-1}$. On the other hand, 
\begin{equation*}
\begin{split}
\sum_{k=1}^N \int_{\US^{n-1}} &|V^{R}(r_{k-1}, \theta)| | |u(r_k, \theta)|^2 - |u(r_{k-1}, \theta)|^2| d\theta \\
&\le  \|V\|_{L^\infty([a,b] \times \US^{n-1})} \sum_{k=1}^N \int_{\US^{n-1}}\int_{r_{k-1}}^{r_k} |\partial_r |u(r, \theta)|^2| dr d\theta \\
&\le \omega_{n-1}(b-a) \|V\|_{L^\infty([a,b] \times \US^{n-1})}  \|\partial_r|u|^2\|_{L^\infty([a,b] \times \US^{n-1})}.
\end{split}
\end{equation*}
Taken together, the previous two estimates show that $r \mapsto \int_{\US^{n-1}} V^{R}(r, \theta) |u(r, \theta)|^2 d\theta$ has locally bounded variation.

We finish by confirming \eqref{technical dist formula}. For any $\varphi \in C_0^\infty(0, \infty)$,
\begin{equation*}
\begin{split}
    -\int_0^\infty &\varphi'(r) \int_{\US^{n-1}} V^R(r,\theta) |u(r, \theta)|^2 d\theta dr \\ &=-\int_{\US^{n-1}} \int_0^\infty \varphi'(r)  V^R(r,\theta) |u(r, \theta)|^2 dr d\theta \\
    &= -\int_{\US^{n-1}} \int_0^\infty \varphi'(r)  V(r,\theta) |u(r, \theta)|^2 dr d\theta \\
    &= - \int_{\US^{n-1}} \int_0^\infty  ((\varphi(r) |u(r, \theta)|^2)' - \varphi(r) 2 \real(\overline{u}(r,\theta)u'(r, \theta)))V(r,\theta)dr d\theta \\
    & =   \int_{\US^{n-1}} \int_0^\infty \varphi'(r) |u(r, \theta)|^2 dV(r, \theta) d\theta \\
    &+2 \int_{\US^{n-1}} \int_0^\infty \varphi(r) \real(\overline{u}(r,\theta)u'(r, \theta)))V(r,\theta) dr d\theta,
    \end{split}
\end{equation*}
which is \eqref{technical dist formula}. Note that to get the first equals sign we used Fubini's theorem.  For the second equals sign we used that for each $\theta \in \US^{n-1}$, $V^R(\cdot,\theta) = V(\cdot, \theta)$ almost everywhere with respect to the measure $dr$. For the last equals sign we used \eqref{dist deriv}.
\section{Phragm\'en-Lindel\"of Theorem} \label{pl appendix}

In this appendix we recall the Phragm\'en Lindel\"of Theorem. Let $f(z)$ be a holomorphic function in a domain $D$ of the complex plane with boundary $\Gamma$. We say that $f(z)$ does not exceed a number $M \ge 0$ in modulus at a boundary point $\zeta \in \Gamma$ if $\limsup_{z \to \zeta, \, z \in D} |f(z)| \le M$.

\begin{theorem}[Phragm\'en Lindel\"of Theorem \cite{em}] \label{pl thm}
Suppose $E \subseteq \Gamma$, and $f$ analytic on $D$ does not exceed $M$ in modulus at any point of $\Gamma \setminus E$. Suppose also there is a function $g(z)$ with the following properties:
\begin{enumerate}
\item $g(z)$ is analytic in $D$,
\item $|g(z)| < 1$ in $D$,
\item $g(z) \neq 0$ in $D$,
\item For every $\sigma > 0$, the function $|g(z)|^\sigma |f(z)|$ does not exceed $M$ is modulus at any $\zeta \in E$.
\end{enumerate}
Under these conditions, $|f(z)| \le M$ everywhere in $D$.
\end{theorem}

\section{Density argument: proof of \eqref{resolv est in sector} and \eqref{resolv est in sector low freq}} \label{density appendix}

In this appendix, we prove \eqref{resolv est in sector} and \eqref{resolv est in sector low freq} as a consequence of

\begin{lemma}
Fix $h, \, s_1 > 0$, $0 < s_2 < 1$, and $z \in \C \setminus [0, \infty)$. Let $P(h)$ be as in \eqref{P} with $V : \R^n \to \R$ obeying \eqref{for self-adjointness} and \eqref{for self-adjointness II} (so that $P(h)$ is self-adjoint with respect to the domain $H^2(\R^n)$). Suppose there exists $C > 0$ so that for all $v \in C^\infty_0(\R^n)$,
\begin{equation} \label{wtd est appendix}
\| \langle x \rangle^{-s_1} v \|^2_{L^2} \le C \| \langle x \rangle^{s_2} (P(h) - z) v \|^2_{L^2}. 
\end{equation}  
Then
\begin{equation} \label{resolv est appendix}
\|  \langle x \rangle^{-s_1} (P(h) - z)^{-1} \langle x \rangle^{-s_2} \|_{L^2 \to L^2} \le C. 
\end{equation}

\end{lemma} 

\begin{proof}
The operator
\begin{equation*}
[P(h), \langle x \rangle^{s_2}]\langle x \rangle^{-s_2} = \left(-h^2( \Delta \langle x \rangle^{s_2}) - 2h^2 (\nabla \langle x \rangle^{s_2}) \cdot \nabla \right) \langle x \rangle^{-s_2}
\end{equation*}
is bounded $H^2(\R^n) \to L^2(\R^n)$. So, for $v \in H^2(\R^n)$ such that $\langle x \rangle^{s_2} v \in H^2(\R^n)$,
 \begin{equation}\label{Czh}
 \begin{split}
\|\langle x \rangle^{s_2}(P(h)-z)v\|_{L^2}  &\le \|(P(h)-z)\langle x \rangle^{s_2} v \|_{L^2} +  \|[P(h),\langle x \rangle^{s_2}]\langle x \rangle^{-s_2}\langle x \rangle^{s_2}v \|_{L^2}
\\& \le C_{z,h} \| \langle x \rangle^{s_2}v \|_{H^2},
\end{split} 
\end{equation}
for some constant $C_{z, h} >0$ depending on $z$ and $h$.

Given $f \in L^2(\R^n)$, the function $ u= \langle x \rangle^{s_2}(P(h)-z)^{-1}\langle x \rangle^{-s_2} f \in H^2(\R^n)$ because 
\begin{equation*}
u = (P(h) - z)^{-1} (f + w), \qquad w =  [P(h), \langle x \rangle^{s_2}]  u,
\end{equation*}
with  $ [P(h), \langle x \rangle^{s_2}] $ being bounded $L^2(\R^n) \to L^2(\R^n)$ since $s_2 < 1$.

Now, choose a sequence $v_k \in C_{0}^\infty$ such that $ v_k \to  \langle x \rangle^{s_2}(P(h)-z)^{-1}\langle x \rangle^{-s_2} f$ in $H^2(\R^n)$. Define $\tilde{v}_k \defeq \langle x \rangle^{-s_2}v_k$. Then, as $k \to \infty$,
\begin{equation*}
\begin{split}
\| \langle x \rangle^{-s_1} \tilde{v}_k - \langle x \rangle^{-s_1} (&P(h)-z)^{-1}\langle x \rangle^{-s_2}f \|_{L^2}  \\
&\le \| v_k - \langle x \rangle^{s_2} (P(h)-z)^{-1}\langle x \rangle^{-s_2}f \|_{H^2} \to 0.
\end{split}
\end{equation*}
Also, applying \eqref{Czh},
\begin{equation*}
\|\langle x \rangle^{s_2}(P(h)-z)\tilde v_k - f\|_{L^2} \le C_{z,h} \|v_k - \langle x \rangle^{s_2} (P(h)-E \pm i \varepsilon)^{-1} \langle x \rangle^{-s_2} f \|_{H^2} \to 0.
\end{equation*} 
Thus \eqref{resolv est appendix} follows by replacing $v$ by $\tilde{v}_k$ in \eqref{wtd est appendix} and sending $k \to \infty$.\\
\end{proof}

\section{Justification of Remark \ref{sharpness remark}} \label{sharpness appendix}

In the setting of Theorem  \ref{nontrap est thm repulsive V}, consider the case of $V = 0$ and $n =3$. In that scenario the integral kernel of $(P(h) - z)^{-1} = (-h^2 \Delta -z)^{-1}$ with $z \in \C \setminus [0, \infty)$ is given by 
\begin{equation*}
R_0(x,y,z) \defeq h^{-2} \frac{e^{i \tfrac{\sqrt{z}}{h}|x -y|}}{4\pi |x -y|}, \qquad \imag \sqrt{z} > 0.
\end{equation*}
We recall why having a bound like \eqref{nontrap est repulsive V} on $\langle \cdot \rangle^{-s_1}(-h^2 \Delta -z)^{-1}\langle \cdot \rangle^{-s_2}: L^2(\R^3) \to L^2(\R^3)$ requires $s_1, \, s_2 > 1/2$. 

Since the norm of an operator and its adjoint coincide, it suffices to show $s_1 > 1/2$ is necessary.  Use $\sqrt{z}$ of the form $\sqrt{z} = E + i\ep$ for $E > 0$ fixed and $\ep > 0$ tending to zero.  Then, as calculated in the proof of \cite[Theorem 3.5]{dz}, for $f \in C^\infty_0(\R^3)$,
\begin{equation} \label{outgoing asymptotic}
\begin{split}
\langle x \rangle^{-s_1} \int_{\R^3} &R_0(x,y,z) f(y) dy \\
&= h^{-2} \frac{\langle x \rangle^{-s_1}}{4\pi |x|} e^{\frac{i}{h}(E+ i\ep)|x|} (\hat{f} \big( \tfrac{E}{h} \tfrac{x}{|x|} \big) + o(1)) + O(|x|^{-2}), \quad \text{as $\ep \to 0^+$ and $|x| \to \infty$.}
\end{split}
\end{equation}
If $f$ is chosen so that $|\hat{f}| > c$ for some $c > 0$ on $\{|x| = E/h\}$, \eqref{outgoing asymptotic} and $s_1 \le 1/2$ imply\\
 $\| \langle x \rangle^{-s_1} \textstyle\int_{\R^3} R_0(x,y,z) f(y) dy  \|_{L^2} \to \infty$ as $\ep \to 0^+$.

Next, supposing $s_1, \, s_2 > 1/2$, we show why a bound like \eqref{low freq est repulsive V}  on $\langle \cdot \rangle^{-s_1}(-h^2 \Delta -z)^{-1}\langle \cdot \rangle^{-s_2}: L^2(\R^3) \to L^2(\R^3)$ requires additionally that $s_1 + s_2 \ge 2$, which is nearly the condition we impose for \eqref{low freq est repulsive V}. Using $\sqrt{z} = i \ep$ for $\ep > 0$ tending to zero, and $f_\eta(y) = \langle y \rangle^{-\eta- \frac{3}{2}}$, $\eta > 0$, we see as in \cite[Proof of Remark 2]{boha10} that
\begin{equation*}
\begin{split}
 \langle x \rangle^{-s_1} \int_{\R^3} &R_0(x,y,z) \langle y \rangle^{-s_2} f(y) \\
 &= h^{-2} \langle x \rangle^{-s_1} \int_{\R^3} \frac{e^{-\frac{\ep}{h}|x-y|}}{4\pi |x- y|} \langle y \rangle^{-s_2} f(y) dy \\
 &\gtrsim  h^{-2} e^{-\frac{3\ep}{2h} |x|} \langle x \rangle^{-s_1-1} \int_{|y| \le \frac{|x|}{2}}  \langle y \rangle^{-s_2- \eta - \frac{3}{2}} dy \gtrsim h^{-2} e^{-\frac{3\ep}{2h} |x|}  \langle x \rangle^{-s_1 - s_2- \eta + \frac{1}{2}}, 
 \end{split}
\end{equation*}
where the implicit constants indicated by $\gtrsim$ are independent of $\ep$ and $\eta$. First sending $\ep \to 0^+$ gives $s_1 + s_2 \ge 2 -\eta$,  but since $\eta > 0$ is arbitrary, we in turn get $s_1 + s_2 \ge 2$.

To see that the $O(|z|^{-\frac{1}{2}} h^{-1})$-dependence of the right side of \eqref{nontrap est repulsive V} is optimal, consider the function $u = e^{i \frac{\sqrt{z}}{h} x_1} \chi$ for nontrivial $\chi \in C^\infty_0(\R^3 ; [0, 1])$. We have
\begin{equation*}
\langle x \rangle^{s}(-h^2 \Delta -z)u = -i \sqrt{z} h \langle x \rangle^{s} \partial_{x_1} \chi - h^2 e^{i \frac{\sqrt{z}}{h}x_1} \langle x \rangle^s \Delta \chi \qefed f,
\end{equation*}
whence $\langle \cdot \rangle^{-s} (-h^2 \Delta - z)^{-1} \langle \cdot \rangle^{-s} f = \langle \cdot \rangle^{-s}u$ and thus, as $h \to 0$,
\begin{equation*}
\frac{\|\langle \cdot \rangle^{-s} (-h^2 \Delta - z)^{-1} \langle \cdot \rangle^{-s} f\|_{L^2}}{\|f \|_{L^2} } = \frac{\| \langle \cdot \rangle^{-s} u\|_{L^2}}{\|f \|_{L^2} } \gtrsim |z|^{-\frac{1}{2}} h^{-1}.
\end{equation*}

Finally, we argue why the $O(h^{-2})$-dependence of the right side of \eqref{low freq est repulsive V} is sharp. As noted before \cite[Proposition 2.4]{gimo74} gives that $R_{0, s_1, s_2}(\lambda) = \langle \cdot \rangle^{-s_1}(-h^2 \Delta -\lambda^2)^{-1}\langle \cdot \rangle^{-s_2}$ ($s_1,\, s_2 > 1/2$, $s_1 + s_2 > 2$), extends continuously from $\imag \lambda > 0$ to $\R$ in the space of bounded opeartors $\R^3 \to \R^3$. In this case, we have,
\begin{equation*}
h^{-2}\big\| \langle x \rangle^{-s_1} \int_{\R^3} \frac{1}{4\pi |x -y|}\langle y \rangle^{-s_2}dy \big \|_{L^2 \to L^2} = \lim_{\ep \to 0} \| R_{0, s_1, s_2}(i\ep)\|_{L^2 \to L^2}.
\end{equation*}
 
\section{Proof of Lemma \ref{lap free resolv square lem}} \label{deriv free resolv appendix}

In this appendix we prove Lemma \ref{lap free resolv square lem}. The proof proceeds in two steps. First we treat the case $|\lambda| \ge 1$, followed by $|\lambda| \le 1$. 
\begin{proof}[Proof of Lemma \ref{lap free resolv square lem}]
Initially we take $|\alpha| = 0$ in \eqref{lap free resolv square}, so may assume without loss of generality that $\imag \lambda > 0$. We treat the $|\alpha| = 1$ case at the end of the proof. Observe that $\tfrac{d}{d\lambda} \langle x \rangle^{-s} (-\Delta - \lambda^2)^{-1} \langle x \rangle^{-s} = 2\lambda \langle x \rangle^{-s} (-\Delta - \lambda^2)^{-2} \langle x \rangle^{-s} $, so we can bound the $L^2(\R^n) \to L^2(\R^n)$ norm of either quantity.  

If $|\lambda| \ge 1 $, begin from
\begin{equation} \label{identity to insert Laplacian}
\begin{split}
-\lambda^2 \langle x \rangle^{-s} (-\Delta - \lambda^2)^{-2} \langle x \rangle^{-s} &= - \tfrac{1}{2} \langle x \rangle^{-s} (-\Delta - \lambda^2)^{-1}(-2\Delta) (-\Delta - \lambda^2)^{-1} \langle x \rangle^{-s}  \\
&+\langle x \rangle^{-s} (-\Delta - \lambda^2)^{-1} \langle x \rangle^{-s}.
\end{split}
\end{equation}
By \eqref{lap A}, the $L^2(\R^n) \to L^2(\R^n)$ norm of the second line of  \eqref{identity to insert Laplacian} is bounded by $C(1 + | \real \lambda|)^{-1}$. So it suffices to investigate the term on the right side of the first line of \eqref{identity to insert Laplacian}. For notational brevity, put $R_0(\lambda) \defeq (-\Delta - \lambda^2)^{-1}$. We show that for all $f \in C^\infty_0(\R^n)$,
\begin{equation} \label{A inv Delta A inv}
\begin{split}
 \langle x \rangle^{-s}  &R_0(\lambda)(-2 \Delta)  R_0(\lambda) \langle x \rangle^{-s} f \\
 &= -\langle x \rangle^{-s} R_0(\lambda) \partial_r (r  \langle x \rangle^{-s} f) + \langle x \rangle^{-s} R_0(\lambda)  \langle x \rangle^{-s}f \\
 &+  \langle x \rangle^{-s} r  \partial_r R_0(\lambda)  \langle x \rangle^{-s}f 
\end{split} 
\end{equation}
Since $s > 3/2$ by \eqref{s restrictions}, the $L^2(\R^n)$-norm of the right side of \eqref{A inv Delta A inv} is bounded by $C \| f\|_{L^2}$, thanks to \eqref{lap A}. So it remains to show \eqref{A inv Delta A inv}.

Recall the well known formula for the Laplacian in polar coordinates,
\begin{equation*}
\Delta = \partial^2_r + (n-1)r^{-1} \partial_r + r^{-2} \Delta_{\US^{n-1}}, 
\end{equation*}
which implies the commutator identity
\begin{equation} \label{commutator identity}
[r \partial_r, \Delta] \defeq r \partial_r (\Delta) - \Delta(r \partial_r) = -2 \Delta. 
\end{equation}
Fix $f \in C^\infty_0(\R^n)$, $g \in L^2(\R^n)$, and put $u \defeq R_0(\lambda) \langle x \rangle^{-s} f \in H^2(\R^n)$. Let $\{u_k\}_{k=1}^\infty \subseteq C^\infty_0(\R^n)$ be a sequence converging to $u$ in $H^2(\R^n)$. Starting from the left side of \eqref{A inv Delta A inv} and applying \eqref{commutator identity},
\begin{equation} \label{long calc 1}
\begin{split}
\langle g, \langle x \rangle^{-s}  R_0(\lambda)(-2 \Delta)  &R_0(\lambda) \langle x \rangle^{-s}f \rangle_{L^2} \\
&= \lim_{k \to \infty} \langle g,  \langle x \rangle^{-s}  R_0(\lambda)[r\partial_r, \Delta]  u_k \rangle_{L^2}\\
&= \langle g, \langle x \rangle^{-s} r \partial_r u \rangle_{L^2} \\
&- \lim_{k \to \infty}  \langle g, \langle x \rangle^{-s}  R_0(\lambda) r \partial_r (-\Delta - \lambda^2) u_k\rangle_{L^2}.
\end{split}
\end{equation}

The purpose of the following calculations is to show that the last line of \eqref{long calc 1} equals\\$-\langle  g, \langle x \rangle^{-s} R_0(\lambda) r \partial_r \langle x \rangle^{-s}   f \rangle_{L^2}$. First, for any $v \in L^2(\R^n)$, $r R_0(\lambda) \langle x \rangle^{-1} v  \in H^1(\R^n)$. This holds because, if we put $w \defeq \langle x \rangle R_0(\lambda) \langle x \rangle^{-1} v$, then $ r R_0(\lambda) \langle x \rangle^{-1} v = r \langle x \rangle^{-1} w$ and 
\begin{equation*}
\begin{gathered}
(-\Delta - \lambda^2) w = [- \Delta , \langle x \rangle] R_0(\lambda) \langle x \rangle^{-1} v + v \implies \\
w = R_0(\lambda)([- \Delta , \langle x \rangle]  R_0(\lambda) \langle x \rangle^{-1} v + v) \in H^2(\R). 
\end{gathered}
\end{equation*}
Furthermore, for any $w, v \in C^\infty_0(\R^n)$, 
\begin{equation*}
\langle w, \partial_r v \rangle_{L^2}  = \langle \partial_r^* w,  v \rangle_{L^2} \defeq (1-n) \langle r^{-1} w, v \rangle_{L^2} - \langle \partial_r w, v \rangle_{L^2}.
\end{equation*}
Therefore, by the density of $C^\infty_0(\R^n)$ in $H^1(\R^n)$, and setting $\tilde{u}_{k} = (-\Delta - \lambda) u_k$, we get
  \begin{equation} \label{long calc 2}
  \begin{split}
\lim_{k \to \infty}  \langle g, \langle& x \rangle^{-s}  R_0(\lambda) (r \partial_r) \tilde{u}_k  \rangle_{L^2}\\
    &= \lim_{k \to \infty}    \langle (\partial_r)^* r R_0(\overline{\lambda}) \langle x \rangle^{-s}  g,   \tilde{u}_k \rangle_{L^2}\\
 &=   \langle  (\partial_r)^* r  R_0(\overline{\lambda}) \langle x \rangle^{-s}  g, \langle x \rangle^{-s}  f \rangle_{L^2}\\
 &= \langle  g, \langle x \rangle^{-s} R_0(\lambda) r \partial_r \langle x \rangle^{-s}   f \rangle_{L^2}.
 \end{split} 
\end{equation}
as desired. Taken together, \eqref{long calc 1} and \eqref{long calc 2} confirm \eqref{A inv Delta A inv}. 

Now we turn to the case $|\lambda | \le 1$, and utilize the integral kernel of the free resolvent, which is given by \cite[Section 3]{jene01},
\begin{equation}
    (-\Delta-\lambda^2)^{-1}(|x-y|)=\frac{1}{2\pi}\left(\frac{-i\lambda}{2\pi|x-y|}\right)^{\frac{n}{2}-1}K_{\frac{n}{2}-1}(-i\lambda|x-y|), \qquad \imag \lambda > 0,
\end{equation}
where $K_\nu(z)$ is the Macdonald function of order $\nu$ \cite[10.27.4, 10.27.5]{dlmf}. Now, if $n=3$, then the integral kernel of $\langle x\rangle^{-s}\frac{d}{d\lambda}(-\Delta-\lambda^2)^{-1}\langle x\rangle^{-s}$ is given by $i(4\pi)^{-1}\langle x\rangle^{-s}e^{i\lambda|x-y|}\langle y\rangle^{-s}$, which has Hilbert-Schmidt norm bounded uniformly in $|\lambda| \le 1$ provided $s>3/2.$ Moving on to $n \ge 4$, by \cite[10.29.2]{dlmf},
\begin{equation*}
    \frac{d}{d\lambda}\left(\frac{-i\lambda}{2\pi|x-y|}\right)^{\frac{n}{2}-1}K_{\frac{n}{2}-1}(-i\lambda|x-y|)=\frac{-(-i)^{\frac{n}{2}}\lambda^{\frac{n}{2}-1}}{(2\pi)^{\frac{n}{2}-1}|x-y|^{\frac{n}{2}-2}}K_{\frac{n}{2}-2}(-i\lambda|x-y|).
\end{equation*}
The Macdonald function satisfies \cite[10.25.3, 10.30.2, 10.30.3]{dlmf}
\begin{equation}
  |K_{\nu}(z)| \le \begin{cases}
   C|z|^{-\nu} & 0 < |z| \le 1, \, \nu > 0, \\
  C |\ln |z|| & 0 < |z| \le 1, \, \nu = 0, \\
   C|z|^{-1/2} &  |z| \ge 1, \,  \real z \ge 0, \, \nu \ge 0,
    \end{cases}
\end{equation}
for $C > 0$ a constant independent of $z$.
Therefore, for $C > 0$ independent of $\lambda$,
\begin{equation} \label{bd deriv kernel low freq}
\begin{split}
\Big| &\frac{\lambda^{\frac{n}{2}-1} \langle x \rangle^{-s} \langle y \rangle^{-s}}{|x-y|^{\frac{n}{2}-2}}K_{\frac{n}{2}-2}(-i\lambda|x-y|)\Big| \\
&\le \begin{cases}
C |\lambda| \langle x \rangle^{-s} \langle y \rangle^{-s} |\ln (|\lambda||x-y|)| \mathbf{1}_{\{|\lambda||x-y|\leq 1\}}+C\frac{|\lambda|^{\frac{n}{2}-\frac{3}{2}} \langle x \rangle^{-s} \langle y \rangle^{-s}}{|x-y|^{\frac{n}{2}- \frac{3}{2}}}  \mathbf{1}_{\{|\lambda||x-y|> 1\}}  & n = 4, \\
C\frac{|\lambda|\langle x \rangle^{-s} \langle y \rangle^{-s}}{|x-y|^{n-4}} \mathbf{1}_{\{|\lambda||x-y|\leq 1\}}+C\frac{|\lambda|^{\frac{n}{2}-\frac{3}{2}}\langle x \rangle^{-s} \langle y \rangle^{-s}}{|x-y|^{\frac{n}{2}- \frac{3}{2}}}  \mathbf{1}_{\{|\lambda||x-y|> 1\}} & n > 4. 
\end{cases}
 \end{split}
\end{equation}
As preparation for the conclusions we draw in the next paragraph, we observe that the first term in line three of \eqref{bd deriv kernel low freq} has the bound, for $|\lambda| \le 1$ and $0 < \alpha \le 1$,
\begin{equation} \label{HS trick}
\frac{|\lambda|\langle x \rangle^{-s} \langle y \rangle^{-s}}{|x-y|^{n-4}} \mathbf{1}_{\{|\lambda||x-y|\leq 1\}} = \frac{|\lambda| |x-y|^{\alpha}\langle x \rangle^{-s} \langle y \rangle^{-s}}{|x-y|^{n-4 + \alpha}} \mathbf{1}_{\{|\lambda||x-y|\leq 1\}} \le \frac{\langle x \rangle^{-s} \langle y \rangle^{-s}}{|x-y|^{n-4 + \alpha}} \mathbf{1}_{\{|\lambda||x-y|\leq 1\}}. 
\end{equation}

In what follows we make repeated use of Lemma \ref{HS lemma} to verify that a given kernel is Hilbert-Schmidt. In \eqref{bd deriv kernel low freq}, the second term in line two and the second term in line three are are uniformly bounded in Hilbert-Schmidt norm for $|\lambda| \le 1$, provided $s > (n + 3)/4$. This also holds for the first term in line two of \eqref{bd deriv kernel low freq} if $s > 3/2$. In addition, to address the first term in line three of \eqref{bd deriv kernel low freq}, we utilize \eqref{HS trick} in combination with Lemma \ref{HS lemma}. Taken together they give that  $\langle x \rangle^{-s} \langle y \rangle^{-s}|x-y|^{-n+4 - \alpha}$ is Hilbert-Schmidt 
\begin{equation*}
\begin{gathered}
\text{when $n = 5$ if $\alpha = 1$ and $s > 3/2$,}\\
\text{when $n = 6$ if $0 < \alpha < 1$ and $s > 2 - (\alpha/2)$, and}\\
\text{when $n =7$ if $0 <\alpha < 1/2$ and $s > 2 - (\alpha/2)$.}
\end{gathered}
\end{equation*}
In particular, when $n = 6$ it is enough to take $s > 3/2$, while when $n = 7$, $s > 7/4$ suffices.

Finally, if $n \ge 8$, the first term in line three is uniformly bounded $L^2(\R^n) \to L^2(\R^n)$ for $|\lambda| \le 1$ provided $s > 3$. This is due to \eqref{HS trick} with $\alpha = 1$ and the Schur test, see Lemma \ref{schur lemma}.

We finish by resolving the $|\alpha| = 1$ case for \eqref{lap free resolv square}. By \eqref{lap free resolv square} in the $|\alpha| = 0$ case, and by \eqref{std elliptic thry}, we need to show $\| \lambda \langle x \rangle^{-s}(-\Delta - \lambda^2)^{-2}  \langle x \rangle^{-s} f \|_{H^2} \le O(1 + |\lambda|) \| f\|_{L^2}$. According to \eqref{recast L2 to H2 bd} below, 
\begin{equation*}
\begin{split}
 \| \lambda& \langle x \rangle^{-s}(-\Delta - \lambda^2)^{-2}  \langle x \rangle^{-s} f \|_{H^2}  \\
 &\le C \| \lambda \langle x \rangle^{-s} (-\Delta - \lambda^2)^{-2}  \langle x \rangle^{-s} f \|_{L^2} +   C\| \lambda \langle x \rangle^{-s}( -\Delta) (-\Delta - \lambda^2)^{-2}  \langle x \rangle^{-s} f \|_{L^2}\\
&= C \| f\|_{L^2} +  C\| \lambda \langle x \rangle^{-s}( -\Delta) (-\Delta - \lambda^2)^{-2}  \langle x \rangle^{-s} f \|_{L^2}.
\end{split}
\end{equation*}
Then use 
\begin{equation*}
\begin{split}
\langle x \rangle^{-s}&( -\Delta) (-\Delta - \lambda^2)^{-2}  \langle x \rangle^{-s} f \\
&= \langle x \rangle^{-s} (-\Delta - \lambda^2)^{-1}  \langle x \rangle^{-s} f + \lambda^2 \langle x \rangle^{-s} (-\Delta - \lambda^2)^{-2}  \langle x \rangle^{-s}f,
\end{split}
\end{equation*}
which in combination with \eqref{lap A}, as well as \eqref{lap free resolv square} in the $|\alpha| = 0$ case yields
\begin{equation*}
 \| \lambda \langle x \rangle^{-s}( -\Delta) (-\Delta - \lambda^2)^{-2}  \langle x \rangle^{-s} f \|_{L^2} \le C(1 + |\lambda|) \| f_{L^2},
\end{equation*}
completing the proof. \\
\end{proof}

\section{Useful lemmas}

\begin{lemma}[{\cite[Proposition 6]{fa67}}] \label{faris lemma}
Let $n \ge 3$. Then,
\begin{equation*}
\| r^{-1} u \|^2_{L^2} \le \Big( \frac{2}{n -2} \Big)^2 \| \nabla u \|^2_{L^2}, \qquad u \in H^1(\R^n). 
\end{equation*} 
\end{lemma}

\begin{lemma}[Schur's test {\cite[Section A.5]{dz}}] \label{schur lemma}
Suppose that $K(x,y)$ is measurable on $\R^n \times \R^n$ and 
\begin{equation*}
\sup_x \int |K(x,y)| dy, \, \sup_y \int |K(x,y)| dx \le C. 
\end{equation*}
Then the linear operator
\begin{equation*}
Tf(x) = \int K(x,y) f(y) dy,
\end{equation*}
obeys the estimate 
\begin{equation*}
\| T f \|_{L^2} \le C \|f \|_{L^2}.
\end{equation*}
\end{lemma}

\begin{lemma}[{\cite{pe24}}] \label{HS lemma} The necessary and sufficient conditions for 
\begin{equation*}
\int_{\R^n} \int_{\R^n} \langle x \rangle^{-s} \langle y \rangle^{-t} |x -y|^{-p} dx dy < \infty,
\end{equation*}
are 
\begin{equation*}
s + p > n, \quad t + p > n, \quad s + p + t > 2n, \quad p < n. 
\end{equation*}
\end{lemma}

\begin{lemma} \label{simple Sobolev est lem}
Suppose $T : L^2(\R^n) \to H^2(\R^n)$ is a bounded operator. For any $s > 0$, there exists $C > 0$ so that 
\begin{equation} \label{recast L2 to H2 bd}
\| \langle x \rangle^{-s} T \|_{L^2 \to H^2} \le C(\| \langle x \rangle^{-s} T \|_{L^2 \to L^2} + \| \langle x \rangle^{-s} \Delta T \|_{L^2 \to L^2}).
\end{equation}
\end{lemma}

\begin{proof}
Let $f \in L^2(\R^n)$ and put $u = Tf$. By the first line of \eqref{std elliptic thry}, there exists $C > 0$, whose precise value may change from line to line, so that 
\begin{equation} \label{apply std elliptic thry}
\| \langle x \rangle^{-s} u \|_{H^2} \le C  \| \langle x \rangle^{-s} u \|_{L^2} + C\| \Delta \langle x \rangle^{-s} u \|_{L^2}, \qquad \tilde{u} \in H^2(\R^n).
\end{equation}
Then use the second line of \eqref{std elliptic thry},
\begin{equation*}
\begin{split}
\| \Delta \langle x \rangle^{-s} u \|_{L^2} &\le \|[\Delta, \langle x \rangle^{-s}] u \|_{L^2} + \|\langle x \rangle^{-s} \Delta  u \|_{L^2} \\
&\le C\|\langle x \rangle^{-s} u \|_{H^1} + \|\langle x \rangle^{-s} \Delta  u \|_{L^2} \\
&\le C \gamma^{-1} \| \langle x \rangle^{-s} u \|_{L^2} +  C\gamma \| \Delta \langle x \rangle^{-s} u \|_{L^2}) +\|\langle x \rangle^{-s} \Delta  u \|_{L^2}, \qquad \gamma > 0.
\end{split}
\end{equation*}
Fixing $\gamma$ small enough yields,
\begin{equation*}
\| \Delta \langle x \rangle^{-s} u \|_{L^2} \le C(\| \langle x \rangle^{-s} u \|_{L^2} +\|\langle x \rangle^{-s} \Delta  u \|_{L^2}),
\end{equation*}
which in combination with \eqref{apply std elliptic thry} implies \eqref{recast L2 to H2 bd}.\\
\end{proof}

\end{document}